\documentclass{amsart}
\usepackage{amssymb,amsfonts}
\usepackage[all,arc]{xy}
\usepackage{enumerate}
\usepackage{amscd}
\usepackage{makecell}
\usepackage{graphicx}
\usepackage{mathrsfs}
\usepackage{amsmath}
\usepackage{amsfonts}
\usepackage{hyperref}
\usepackage{booktabs}
\usepackage{cleveref}
\usepackage{tikz}
\usepackage[utf8]{inputenc}
\usepackage[margin=1.3in]{geometry}
\usepackage[space]{grffile}

\usepackage{graphicx}
\usepackage{color}

\newcommand{\R}{\mathbb{R}}

\newcommand{\Q}{\mathbb{Q}}

\newcommand{\Z}{\mathbb{Z}}

\newcommand{\ra}{\rightarrow}

\newcommand{\lip}{\operatorname{Lip}}

\newtheorem{thm}{Theorem}[section]
\newtheorem{cor}[thm]{Corollary}
\newtheorem{prop}[thm]{Proposition}
\newtheorem{lem}[thm]{Lemma}

\newtheorem{defn}[thm]{Definition}

\title{Flexible exponents of non-geometric 3-manifolds}

\author{Jianfeng Lin}
\address{Department of Mathematics Science, Tsinghua University, Beijing, 100080, CHINA}
\email{linjian5477@mail.tsinghua.edu.cn}

\author{Hongbin Sun}
\address{Department of Mathematics, Rutgers University - New Brunswick, Hill Center, Busch Campus, Piscataway, NJ 08854, USA}
\email{hongbin.sun@rutgers.edu}
\author{Zhongzi Wang}
\address{Department of Mathematical Sciences, Peking University, Beijing 100871 CHINA}
\email{wangzz22@stu.pku.edu.cn}

\date{\today}
\thanks{H.S. is partially supported by the Simons Collaboration Grant 615229.}

\begin{document}
\bibliographystyle{amsalpha}
\maketitle

\begin{abstract}  
 A classical question in quantitative topology is to bound the mapping degree  $\operatorname{deg}(f)$ in terms of its Lipchitz constant $\operatorname{Lip}(f)$. For a closed, orientable, Riemannian manifold $M$, the flexible exponent $\alpha(M)$ is the infimum of  $\alpha\geqslant 0$ such that $|\deg (f)|\leqslant C\cdot (\lip (f))^\alpha$ holds for any Lipschitz map $f:M\ra M$. 
  
    For a geometric 3-manifold $M$ in the sense of Thurston,   $\alpha(M)$ is determined in \cite{DLWWW}.
In this paper, we determine $\alpha(M)$ for non-geometric 3-manifolds. 
\end{abstract}

\tableofcontents

\section{Introduction}


We assume all manifolds in this paper are connected, compact, and orientable, unless otherwisely stated.

For a closed, orientable manifold $M$, we use $D(M)$ to denote the set of degrees of self-maps of $M$. Following Gromov, a closed, orientable manifold $M$ is called inflexible (or flexible) if and only if $D(M)$ is finite (or infinite). 
Inspired by a recent work of Berdnikov-Guth-Manin \cite{BGM}, the flexible exponent $\alpha(M)$ was defined in \cite{DLWWW}, which is a numerical invariant of $M$. Roughly speaking, the flexible exponent measures how effectively a self-map wraps $M$ to itself.

Let $M$ and $N$ be compact Riemannian manifolds.
For a map $f:M\to N$, we define the Lipschitz constant of $f$ as
\[
    \lip (f):=\sup_{u,v\in M} \frac{d_N(f(u),f(v))}{d_M(u,v)}\in [0,+\infty].
\]
We say that $f$ is an $L$-Lipschitz map for some positive number $L$ if $\lip (f)\leqslant L<+\infty$.


\begin{defn}
    Suppose $M$ is a closed, orientable Riemannian $n$-manifold. The flexible exponent of $M$ is defined as 
 \[\alpha(M):=\inf\{\alpha\geq 0\mid \exists C>0 \text{ s.t. } |\operatorname{deg}(f)|\leq C\cdot (\operatorname{Lip}(f))^{\alpha} \text{ for all }f:M\to M\}.\]    
    \end{defn}

 It is known  that  $\alpha(M)\in[0,\dim M]$ always holds, and $\alpha(M)>0$ if and only if $M$ is flexible, see \cite{DLWWW}. For a simply-connected closed manifold $M$, it is known by \cite[Theorem A, Theorem D]{BGM} that $\alpha(M)=\dim M$ if and only if $M$ is formal. 


For any closed orientable geometric 3-manifold $M$,  $\alpha(M)$ is determined  in \cite{DLWWW}, and a summary of results in \cite{DLWWW} can be found below.
 We will determine  $\alpha(M)$ for non-geometric 3-manifolds $M$ in this paper.

  
 By Kneser--Milnor's prime decomposition theorem and Papakyriakopoulos's sphere theorem, any closed orientable 3-manifold $M$ other than $S^3$ has a prime decomposition (unique up to orders and homeomorphisms)
$$M=(\#_{i=1}^mM_i)\#(\#_{j=1}^n N_j)\#(\#^kS^2\times S^1),$$ 
where each $M_i$ is aspherical, and each $N_j$ has finite fundamental group.

There are eight homogeneous simply-connected complete Riemannian 3-manifolds:
$$\mathbb{H}^3, \widetilde{PSL}(2,\R), \mathbb{H}^2\times \mathbb{E}^1, Sol, Nil, \mathbb{E}^3, S^3, S^2\times \mathbb{E}^1.$$
We say that a closed orientable 3-manifold $M$ is geometric, in the sense of Thurston,  if it supports one of above eight geometries, and otherwise $M$ is non-geometric.

In Thurston's geometrization picture (confirmed by Perelman), 
 it is proved that  a closed orientable 3-manifold $M$ has infinite $D(M)$ if and only if $M$ belongs to one of the following topological families (see \cite[Corollary 4.3]{Wang}).
\begin{itemize}
    \item [(i)] $M$ is covered by a torus bundle over the circle, or

    \item [(ii)] $M$ is covered by $\Sigma\times S^1$ for some closed orientable surface $\Sigma$ with genus $>1$, or

    \item [(iii)] each prime factor of $M$ is covered by $S^3$ or $S^2\times S^1$.
\end{itemize}

 
Therefore we have the following criterion.
    A closed orientable 3-manifold $M$ has positive flexible exponent if and only if $M$ belongs to one of the following two geometric families.
    \begin{itemize}
        \item $M$ admits either $S^3,\ \mathbb E^3,\ S^2\times \mathbb E^1,\ \mathbb H^2\times \mathbb E^1$, Nil, or Sol geometry, or
        \item $M$ is non-geometric, and each prime factor of $M$ admits either $S^3$ or $S^2\times \mathbb E^1$ geometry.
    \end{itemize}

For closed orientable geometric 3-manifolds, the value of $\alpha(M)$ was computed in \cite{DLWWW}:
 \[
      \begin{tabular}{c|c|c|c|c|c}
            \toprule
                  Geometry of $M$  & \makecell{$S^3,\ \mathbb E^3,\ S^2\times \mathbb E^1$} & Nil  & Sol  & \makecell{$\mathbb H^2\times\mathbb E^1$} & $\mathbb H^3, \widetilde {SL_2}$ \\ \midrule
                $\alpha(M)$ & $3$ & $\frac 83$ & $2$  & $1$ & $0.$\\
            \bottomrule
    \end{tabular}
    \]


  

In this paper, we compute  $\alpha(M)$ for all closed, orientable, non-geometric 3-manifolds.
       
       \begin{thm}\label{non-geometric} Suppose $M$ is a closed, orientable, non-geometric 3-manifold. 
        Then 
 \begin{enumerate}        
  \item       $\alpha(M)= 2$ if $M$ is a nontrivial connected sum, and  each prime factor admits the geometry of  
    $S^3$ or $S^2\times \mathbb{E}^1$, 
     \item     $\alpha(M)=0$ otherwise.
    \end{enumerate} 
      
         \end{thm} 
         
We only need to prove Theorem \ref{non-geometric} (1), since the $3$-manifolds in (2) has finite $D(M)$, by  \cite[Corollary 4.3]{Wang}.
         
In Section \ref{basics}, we review and prove some basic properties of the flexible exponent. In Section \ref{upperbound}, we prove that $\alpha(M)\leq 2$ for $M$ as in Theorem \ref{non-geometric} (1).  In Section \ref{finish}, we prove that $\alpha(M)\geq 2$ for $M$ as in Theorem \ref{non-geometric} (1) based on Theorem \ref{technical}, which is a major technical result that constructs non-zero degree self-maps of $\#^kS^2\times S^1$ with controlled Lipschitz constants. Theorem \ref{technical}  will be proved in Section \ref{technicalsection}.
Certain delicate designs of graph packings in $\#^k S^2\times S^1$
make the construction  of   maps with controlled Lipschitz constants   possible.  

Up to some well-known facts, this paper is self-contained.

\section{The flexible exponent}\label{basics}

We start with some basic properties of Lipschitz maps.
 
\begin{prop}\label{LipschitzBasicProperties}
    Let $X$, $Y$, $Z$ be Riemannian manifolds.
    \begin{enumerate}
        \item If $f:X\ra Y$ and $g:Y\ra Z$ are Lipschitz maps, then
        \[
            \lip (g\circ f)\leq \lip (f)\cdot \lip (g).
        \]
        \item For a differentiable map  $f:X\ra Y$, we use $df:TX\ra TY$ to denote the tangent map of $f$. We consider the operator norm
        \[
            \|df\|:=\sup_{p\in X} \|df_p\|,\quad\text{where\quad } \|df_p\|:=\sup_{v\in T_pX,\ \|v\|=1} \|df_p (v)\|_{T_{f(p)}Y}.
          \]   Then we have
        \[
            \lip (f)=\|df\|.
        \]
        \item If a map $f:X\to Y$ is differentiable away from an embedded codimension-$1$ subcomplex $Z\subset X$, then $$\lip(f)=\sup _{p\in X\setminus Z}\|df_p\|.$$ 
    \end{enumerate}
\end{prop}

We will call a map $f:X\to Y$ as in Proposition \ref{LipschitzBasicProperties} (3) a piecewise smooth map.

The 3-manifold $M$ that we consider in this paper do not have naturally associated geometries, since they are non-geometric.
To estimate $\alpha(M)$, we must put a Riemannian metric on $M$. The next lemma shows that $\alpha(M)$ is independent of 
the choice of Riemannian metrics.

\begin{lem}\label{homotopy}
    Suppose two closed connected orientable Riemannian $n$-manifolds $M_1,\ M_2$ are homotopy equivalent to each other, then $\alpha(M_1)=\alpha(M_2)$.
    \end{lem} 
    
    Another description of $\alpha(M)$ makes the proof of Lemma \ref{homotopy} more directly.
    For a closed orientable Riemannian $n$-manifold $M$  and number  $L>0$, define 
    \[
        P_{M}(L):=\sup \{|\deg (f)|\mid f:M\ra M \text{ is a Lipschitz map with }\lip (f)\leqslant L\}.
    \]
    
  It is clear that $P_{M}$ is a non-negative function and is also non-decreasing, and
  \[
        \alpha(M)=\inf\{\alpha\geqslant 0\mid P_M(L)\leqslant C\cdot L^\alpha \text{ for some $C>0$}\}.
    \]

\begin{proof}[Proof of Lemma \ref{homotopy}]
 

 Let $f_1: M_1\to M_2$ and $f_2: M_2\to M_1$ be maps realizing the homotopy equivalence. 
By smooth approximation, we can assume that  $f_1$ and $f_2$ are smooth maps of degree $\pm 1$.
Since $M_1$, $M_2$ are compact, we have 

\[\lip (f_1), \lip (f_2)\le  C\] for some constant $C>0$.

{\bf Claim:} $P_{M_1} (L)\le P_{M_2}(C^2L)$. 

Proof: By definition, for any $\epsilon>0$, there exists a map $f: M_1\to M_1$ such that \[\lip (f)\le L, \,\, |\deg (f)|>P_{M_1} (L)-\epsilon.\]
Let $f'=f_1\circ f\circ f_2 : M_2\to M_2$. Then it is easy to see that 
    \[
        \lip  (f')\leqslant \lip (f_1)\cdot \lip (f) \cdot  \lip (f_2) \le C^2 L, \quad  |\deg (f)|=|\deg (f')|.
    \]
    Then  the Claim follows by 
    \[P_{M_1} (L)-\epsilon<|\deg (f)|=|\deg (f')|   \le P_{M_2}(\lip (f')) \le P_{M_2}(C^2L).\]

    The Claim implies that $\alpha(M_1)\le \alpha(M_2)$.
    By symmetry, $\alpha(M_1)\ge \alpha(M_2)$ also holds. So we have $\alpha(M_1)= \alpha(M_2)$.
     \end{proof}


Now we state several facts as lemmas below. The first two will be used to prove $\alpha(M)\le 2$, and the third will be used to
prove $\alpha(M)\ge 2$, for $M$ in Theorem \ref{non-geometric} (1).

\begin{lem}\label{LiftingUpperBoundsExponent}
    Let $p:\tilde M\ra M$ be a finite cover between closed orientable manifolds. Suppose that any non-zero degree self-map $f:M\ra M$ can be lifted to $\tilde f :\tilde M\ra \tilde M$, then $\alpha(M)\leqslant\alpha(\tilde M)$.
\end{lem}
\begin{proof}
     Fix a Riemannian metric for $M$ and let $\tilde g$ be the pull-back Riemannian metric on $\tilde M$ by the covering map. Then for any non-zero degree self-map $f:M\ra M$ and its lift $\tilde f :\tilde M\ra \tilde M$ we have
     \[
        \lip (f)=\lip (\tilde f),\quad \deg (f)=\deg(\tilde f).
     \]
     Then it follows from the definition that $\alpha(M)\leqslant \alpha(\tilde M)$.
\end{proof}


\begin{lem}\label{HomologyControlsFlexibleExponent}
    Suppose $M$ is a closed, orientable, differentiable $n$-manifold, and we fix a Riemannian metric on $M$. Then there exists a positive constant $C>0$, such that for any  non-zero degree differentiable self-map $f:M\ra M$, the following statements hold.
    \begin{enumerate}
        \item If the $k$-th betti number $\beta_k:=\dim_\Q H_k(M;\Q)$ is positive for some integer $k$ and let $f_*:H_k(M;\Q)\ra H_k(M;\Q)$ be the induced linear map, then we have
        \[
            |\det(f_*)|\leqslant C\cdot (\lip (f))^{k\cdot \beta_k},
        \]
        where $\det(f_*)$ is the determinant of $f_*$.
        \item Moreover, if $f$ induces an isomorphism on $H_k(M;\Z)/\operatorname{Tor}$, then we have
    \[
        |\deg (f)|\leqslant C\cdot (\lip (f))^{n-k}.
    \]
    In particular, $\alpha(M)\leqslant n-k$ holds.
    \end{enumerate}
\end{lem}
\begin{proof}
    We fix a fundamental classes $[M]\in H_n(M;\Z)$. The cup product induces a non-degenerate bilinear pairing 
    \[
        x(u,v):=\langle u\smile v,[M]\rangle,\quad u\in H^k(M;\Q),\ v\in H^{n-k}(M;\Q)
    \]
    on the rational cohomology groups. By Poincar\'e duality and the universal coefficient theorem,  the vector spaces $H^{k}(M;\Q)$ and $H^{n-k}(M;\Q)$ have the same $\Q$-dimension, which equals $\beta_k$. We choose a rational basis $\eta_1,\ldots,\eta_{\beta_k}$ for $H^{k}(M;\Q)$ and let $\omega_1,\ldots,\omega_{\beta_k}$ be the dual basis for $H^{n-k}(M;\Q)$, such that
    \[
        x(\eta_i,\omega_j)=\delta_{ij},\quad 1\leqslant i,j\leqslant \beta_k.
    \]
In the following, we use the same notation to denote a de Rham cohomology class and a closed differentiable form representative.
    
By Thom's realization theorem \cite{Thom}, for any $j$, there exists a closed, oriented, immersed $k$-dimensional submanifold $X_j$ (possibly disconnected) of $M$ and an integer $n_j$, such that $\frac 1{n_j}[X_j]$ is Poincar\'e dual to $\omega_j$. (We need $\omega_j\in H^{n-k}(M;\mathbb{Q})$ to guarantee the existence of $X_j$.) Then pairings with $\omega_j$ can be thought of as integrations on these submanifolds, i.e. for any cohomology class $u\in H^{k}(M;\Q)$, we have
\[
    x(u,\omega_j)=\langle u\smile \omega_j,[M]\rangle=\frac{1}{n_j}\int_{X_j}u.
\]

    Under the bases $\{\eta_i\}$ and $\{\omega_j\}$, 
    let $A_1:H^k(M;\Q)\ra H^k(M;\Q)$ and $A_2:H^{n-k}(M;\Q)\ra H^{n-k}(M;\Q)$ be the matrices of the linear maps induced by $f$. Then there is an upper bound for the entries of $A_1$: 
    \[
     |(A_1)_{i,j}|=|x(f^*\eta_i,\omega_j)|=\Big|\frac{1}{n_j}\int_{X_j}f^*\eta_i\Big| \leqslant (\lip (f))^k\cdot\frac{1}{n_j}\int_{X_j}\|\eta_i\|\ d\text{vol}_k\leq C\cdot (\text{Lip}(f))^k.
    \]
    
    Here $\|\eta\|$ denotes the norm of a differential form induced by the Riemannian metric on $M$, and $d\text{vol}_k$ donotes the $k$-dimensional volume form on $X_j$ induced by the Riemannian metric. The constant $C>0$ is the maximum of 
    $$\{\frac{1}{n_j}\int_{X_j}\|\eta_i\|d\text{vol}_k\ |\ i,j=1,\cdots,\beta_k\},$$, which depends on the Riemannian metric, the forms $\eta_i$, the $k$-dimensional submanifolds $X_j$, and the integers $n_j$. It is important to note that $C$ is independent of the map $f$.
   
    By the universal coefficient theorem, the matrix representing $f_*:H_k(M;\mathbb{Q})\to H_k(M;\mathbb{Q})$ is the transpose of $A_1$ and hence we have
    \[
        |\det(f_*)|=|\det (A_1)|\leqslant \beta!\cdot C^{\beta_k}\cdot (\lip (f))^{k\cdot \beta_k}.
    \]
    Up to replacing $\beta_k!\cdot C^{\beta_k}$ by $C$, this proves the first statement.
    
    For the second statement, the same proof also applies to $A_2$, so we have $|\det (A_2)|\leqslant \beta!\cdot C^{\beta_k}\cdot (\lip (f))^{{(n-k)}\cdot \beta_k}$ (for a possibly different constant $C$ independent of $f$). If $f_*$ induces an isomorphism on $H_k(M;\Z)/\operatorname{Tor}$, then $\det(A_1)=\pm 1$.
    It is clear that $x(f^*(u),f^*(v))=\pm\deg (f)\cdot x(u,v)$, which implies that $A_1^t\cdot A_2=\pm \deg (f)\cdot I$ and
    \[
        \det (A_1)\cdot \det (A_2)=\pm(\deg (f))^{\beta_k}.
\]
    So we have
    \[
        |\deg (f)|=|\det (A_2)|^{1/\beta_k}\leqslant (\beta!)^{\frac{1}{\beta_k}}\cdot C\cdot (\lip (f))^{n-k}.
    \]
    Up to changing the constant, we have $|\deg(f)|\leq C\cdot (\text{Lip}(f))^{n-k}$.
    
By a result attributed to Shoen and Uhlenbeck (see the introductions of \cite{Be, Ha}, with $W^{1,p}=W^{1,\infty}$), any Lipschitz map between Riemannian manifolds can be approximated by differentiable maps, with arbitrarily close Lipschitz constants. So the above argument for differentiable maps implies that $\alpha(M)\leq n-k$ holds.
\end{proof}


\section{An upper bound of flexible exponents of non-geometric 3-manifolds}\label{upperbound}


By the discussion in the introduction, Theorem \ref{non-geometric} (2) follows from \cite[Corollary 4.3]{Wang}. So it remains to prove Theorem \ref{non-geometric} (1), which is restated as the following result. 

\begin{thm}\label{nongeometric}
    Suppose $M$ is a closed, orientable, non-geometric 3-manifold and  \[
        M=(\#^mS^2\times S^1)\#(\#_{i=1}^l P_i),
    \] where each $P_i$ is covered by $S^3$ and $|\pi_1(P_i)|>1$.
  Then $\alpha(M)=2$.
    
    \end{thm}
    
 We will prove that $\alpha(M)\leq 2$ in this section, and will prove $\alpha(M)\geq 2$ in the next two sections. We first give a more explicit description of the $3$-manifold $M$ in Theorem \ref{nongeometric}.
    
\begin{lem}\label{classify}
Any closed, orientable, non-geometric $3$-manifold $$M=(\#^mS^2\times S^1)\#(\#_{i=1}^l P_i)$$ in Theorem \ref{nongeometric} belongs to one of the following two families:
\begin{enumerate}
\item $M=\#^mS^2\times S^1$ with $m\geq 2$,
\item $M=(\#^mS^2\times S^1)\#(\#_{i=1}^l P_i)$ such that $l\geq 1$, $m+l\geq 2$, and $M\ne \mathbb{R}P^3\#\mathbb{R}P^3$.
\end{enumerate}
\end{lem}

\begin{proof}
If $M$ is a connected sum of $m$ copies of $S^2\times S^1$, since $\#^0S^2\times S^1=S^3$ and $\#^1S^2\times S^1=S^2\times S^1$ are geometric, we must have $m\geq 2$. So $M$ belongs to case (1).

For an $M$ not in case (1), we have $l\geq 1$. If $l=1$, then we must have $m\geq 1$, otherwise $M=P_1$ is geometric. So we have $m+l\geq 2$. Moreover, $\mathbb{R}P^3\# \mathbb{R}P^3$ is geometric since it is doubly covered by $S^2\times S^1$. So any $M$ in Theorem \ref{nongeometric} and not in case (1) belongs to case (2). 
\end{proof}

For any $3$-manifold $M$ in Lemma \ref{classify} (2), we want to construct a finite cover that is homeomorphic to $\#^k S^2\times S^1$ with $k\geq 2$. This follows from the fact that $M$ is a non-geometric $3$-manifold, but we give a more direct proof here.

\begin{lem}\label{cover}
 Suppose $M$ is a 3-manifold and  \[
        M=(\#^mS^2\times S^1)\#(\#_{i=1}^l P_i).
    \] where each $P_i$ is covered by $S^3$ and $|\pi_1(P_i)|>1$. We also assume that $l\geq 1$, $m+l\geq 2$, and $M\ne \mathbb{R}P^3\# \mathbb{R}P^3$. Then $M$ has a finite regular cover $M'$ that is homeomorphic to $\#^k S^2\times S^1$ with $k\geq 2$.
 \end{lem}
 
 \begin{proof}
We have $$
        \pi_1(M)=(*^m \mathbb{Z})*(*_{i=1}^l \pi_1(P_i)),$$
    and we consider the group homomorphism
    $$ \Phi:\pi_1(M)\longrightarrow \Z^{m}\times \pi_1(P_1)\times\cdots\times \pi_1(P_l)\longrightarrow  \pi_1(P_1)\times\cdots\times \pi_1(P_l),$$
    where $\Phi$ is the composition of two obvious homomorphisms. Note that the restriction of $\Phi$ sends $\pi_1(P_i)<\pi_1(M)$ to $\pi_1(P_i)<\pi_1(P_1)\times \cdots \times \pi_1(P_l)$ by identity.
    
    The kernel $H:=\ker \Phi$ is a finite index normal subgroup of $\pi_1(M)$. By the Kurosh Theorem in group theory (\cite{ScW}), we have
$$H\cong  (*^k \mathbb{Z})*(*_{j=1}^a C_j),$$
where each  $C_j$ conjugates to some $C<\pi_1(P_i)<\pi_1(M)$ for some $i$. Since $C< \ker \Phi$ and  $\Phi| : \pi_1(P_i)\to  \pi_1(P_i)$ is the identity, $C$ must be the trivial group and so is $C_j$.  We conclude that $H$ is a free group of rank $k< \infty$, and we want to check $k\geq 2$.

Let $\pi: M'\to M$ be the finite regular cover of $M$ corresponding to $H<\pi_1(M)$, then $M'$ is homeomorphic to $\#^k S^2\times S^1$. Let $D=\text{deg}(\pi)=\Pi_{i=1}^l|\pi_1(P_i)|$. Since $l\geq 1$ and $|\pi_1(P_i)|\geq 2$, $D\geq 2$ holds. We consider $M$ as a union of $m$ copies of $S^2\times S^1\setminus \text{int}(B^3)$, one $P_i\setminus \text{int}(B^3)$ for each $i=1,\cdots,l$, and one $S^3\setminus \bigcup^{m+l}\text{int}(B^3)$, pasting along $(m+l)$ copies of $S^2$. 
 
 In $M'$, each $S^2\times S^1\setminus \text{int}(B^3)$ lifts to $D$ copies of $S^2\times S^1\setminus \text{int}(B^3)$, each $P_i\setminus \text{int}(B^3)$ lifts to $\frac{D}{|\pi_1(P_i)|}$ copies of $S^3\setminus \bigcup^{|\pi_1(P_i)|}\text{int}(B^3)$, the $S^3\setminus \bigcup^{m+n}\text{int}(B^3)$ lifts to $D$ copies of $S^3\setminus \bigcup^{m+n}\text{int}(B^3)$, pasting along $D(m+l)$ copies of $S^2$. The corresponding dual graph of $M'$ has 
 $$(m+1)D+\sum_{i=1}^l \frac{D}{|\pi_1(P_i)|}$$ 
 vertices and $D(m+l)$ edges, with $Dm$ vertices correspond to punctured $S^2\times S^1$ and remaining vertices  correspond to punctured $S^3$.
 
The Euler characteristic of this graph is $(1-l)D+\sum_{i=1}^l\frac{D}{|\pi_1(P_i)|}$, so we have 
\begin{equation}
\begin{aligned}
k&=Dm+\Big[1-(1-l)D-\sum_{i=1}^l\frac{D}{|\pi_1(P_i)|}\Big]= 1+D(m-1)+\sum_{i=1}^lD(1- \frac{1}{|\pi_1(P_i)|}).
\end{aligned}
\end{equation}

If $m\geq 1$, since $l\geq 1$, $D\geq 2$, and $|\pi_1(P_i)|\geq 2$, we have 
$$k\geq 1+\sum_{i=1}^lD(1- \frac{1}{|\pi_1(P_i)|})\geq 1+\sum_{i=1}^l\frac{1}{2}D=1+\frac{D}{2}l\geq 2.$$ 
If $m=0$, then $l\geq 2$, and we have
$$k\geq 1-D+\sum_{i=1}^lD(1- \frac{1}{|\pi_1(P_i)|})\geq 1-D+\frac{D}{2}l\geq 1.$$
The equality holds if and only if $l=2$ and $|\pi_1(P_1)|=|\pi_1(P_2)|=2$, which correspond to the case of $M=\mathbb{R}P^3\# \mathbb{R}P^3$. Since $\mathbb{R}P^3\# \mathbb{R}P^3$ is excluded in this lemma and $k$ is an integer, $k\geq 2$ still holds. The proof of this lemma is done.

 \end{proof}


    
  
  Now we are ready to prove that $\alpha(M)\leq 2$ for $M$ in Theorem \ref{nongeometric}.
    
    \begin{prop}\label{upper bound}
    Suppose $M$ is a closed, orientable, non-geometric 3-manifold and  \[
        M=(\#^mS^2\times S^1)\#(\#_{i=1}^l P_i).
    \] where each $P_i$ is covered by $S^3$.
  Then $\alpha(M)\le 2$ holds.
    
    \end{prop}
\begin{proof} 
We first show that $M$ has a finite characteristic cover $\tilde{M}$ that is homeomorphic to the connected sum of at least two copies of $S^2\times S^1$. If $M$ is a connected sum of $S^2\times S^1$, we take $\tilde{M}=M$. Otherwise, Lemma \ref{cover} provides us a finite cover $M'$ of $M$ such that $M'=\#^k S^2\times S^1$ with $k\geq 2$. Then we take a further finite cover $\tilde{M}$ of $M'$, such that $\tilde{M}\to M$ is a characterisitic finite cover.

    By \cite[Corollary 4.1]{Wang}, any non-zero degree self-map of $\tilde{M}$ must induce an isomorphism on the fundamental group, hence an isomorphism on the first homology. By Lemma \ref{HomologyControlsFlexibleExponent} (2), we know that $\alpha(\tilde{M})\le 2$ holds. 
For any non-zero degree self-map $f:M\ra M$, again by \cite[Corollary 4.1]{Wang}, $f_*:\pi_1(M)\ra \pi_1(M)$ must be an isomorphism. So $f_*$ preserves the characteristic subgroup  $\pi_1(\tilde{M})<\pi_1(M)$, and $f$ lifts to a map $\tilde{f}:\tilde{M}\ra \tilde{M}$. Since the choice of $f$ is arbitrary, Lemma \ref{LiftingUpperBoundsExponent} implies that $\alpha (M)\leq \alpha(\tilde{M})\le 2$ holds.
    \end{proof}

\section{A lower bound of flexible exponents of non-geometric $3$-manifolds}\label{finish}

The genus-$k$ handlebody $H_k$ is a primary object in 3-manifold topology, as shown in Figure \ref{Figure1}. 
Recall that $M_k=\#^k S^2\times S^1$ has its standard genus-$k$ Heegaard decomposition, which is obtained by taking two copies of the genus-$k$ handlebody $H$ and $H'$, and pasting their boundaries by the identity map.

\begin{figure}[htbp]
    \centering
    \def\svgwidth{4in} 
   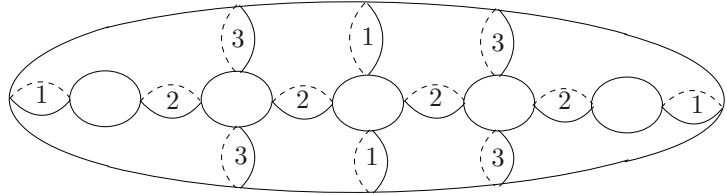
    \caption{The handlebody $H_k$ with $k=5$. Numbered discs decompose $H_k$ to $Y$-shaped solids.}
     \label{Figure1}
\end{figure}

Below  is the main technical result in this paper, whose proof will be given in the next section.
\begin{thm}\label{technical}
Let $M_k=H\cup H'$ be a standard genus-$k$ Heegaard decomposition of $\#^k S^2\times S^1$, and we equip $M_k$ with any Riemannian metric $g$. Then there exists a constant $C>1$ (depending on the metric $g$), such that for any positive integer $n$, there exists a piecewise smooth map $F_n:M_k\to M_k$ that satisfies the following conditions.
\begin{enumerate}
\item $\text{deg}(F_n)=-4n^2-4n=1-(2n+1)^2$.
\item $F_n$ is a $(Cn)$-Lipschitz map.
\item $F_n|_{H'}=id_{H'}$.
\end{enumerate}
\end{thm}



At first, for $M=\#^k S^2\times S^1$ with $k\geq 2$, the result that $\alpha(M)\geq 2$ is an easy consequence of Theorem \ref{technical}.

\begin{lem}\label{easy}
For $M=\#^k S^2\times S^1$ with $k\geq 2$, we have $\alpha(M)\geq 2$.
\end{lem}

\begin{proof}
We take a Riemannian metric $g$ on $M$. Then Theorem \ref{technical} gives us a sequence of $(Cn)$-Lipschitz maps $F_n:M\to M$ such that $\text{deg}(F_n)=-4n^2-4n$.

Then for any $\alpha\in \mathbb{R}$, if $$4n^2+4n=|\text{deg}(F_n)|\leq C'\cdot (\text{Lip}(F_n))^{\alpha}\leq C'\cdot (Cn)^{\alpha}$$ holds for some $C'$ and all $n$, we must have $\alpha\geq 2$. So $\alpha(M)\geq 2$ holds.
\end{proof}

Now we prove that $\alpha(M)\geq 2$ holds for $M$ as in Lemma \ref{classify} (2).

\begin{thm}\label{sphericalsummand}
Suppose $M=(\#^m S^2\times S^1)\#(\#_{i=1}^l P_i)$ as in Lemma \ref{classify} (2), then $\alpha(M)\geq 2$ holds.
\end{thm}

\begin{proof}
Our $M$ satisfies the assumption of Lemma \ref{cover}, so there exists a finite regular cover $\pi:\tilde{M}\to M$ that is homeomorphic to $\#^kS^2\times S^1$ for some $k\geq 2$. We denote the deck transformation group by $\Gamma$. 

Let $\tilde{M}=H\cup H'$ be one standard genus-$k$ Heegaard decomposition of $\tilde{M}$, then $H$ is a tubular neighborhood of a $3$-valence graph $G\subset H \subset \tilde{M}$. Since $G\subset \tilde{M}$ is a graph in a $3$-manifold, and $\Gamma$ is finite, we can perturb $G$ such that $\gamma\cdot G\cap G=\emptyset$ holds for any $\gamma\in \Gamma\setminus \{e\}$.  So there exists a small neighborhood $H''$ of $G$ in $\tilde{M}$ such that $\gamma\cdot H''\cap H''=\emptyset$ holds for any $\gamma\in \Gamma\setminus \{e\}$, and $H''$ is homeomorphic to the genus-$k$ handlebody. Note that for any $\gamma\in \Gamma$, $\tilde{M}=\gamma\cdot H''\cup (\tilde{M}\setminus \text{int}(\gamma \cdot H''))$ is also a standard genus-$k$ Heegaard decomposition of $\tilde{M}=\#^kS^2\times S^1$.


We equip $M$ with a Riemannian metric $g$, and it pulls back to a $\Gamma$-invariant metric $\tilde{g}$ on $\tilde{M}$.
By Theorem \ref{technical}, there exists a constant $C>1$, such that for any positive integer $n$, there is a $(Cn)$-Lipschitz map $F_n:\tilde{M}\to \tilde{M}$, satisfying $\text{deg}(F_n)=1-(2n+1)^2$ and $F_n|_{\tilde{M}\setminus H''}=id_{\tilde{M}\setminus H''}$.

Now we use $F_n$ to construct a $\Gamma$-equivariant map $G_n:\tilde{M}\to \tilde{M}$, as 
\begin{equation*}
\begin{aligned}
G_n(x)=\begin{cases}
x & \text{if\ } x\in \tilde{M}\setminus \bigcup_{\gamma\in \Gamma}\text{int}(\gamma \cdot H''),\\
\gamma \cdot F_n(\gamma^{-1}\cdot x) & \text{if\ } x\in \gamma \cdot H'' \text{\ for\ some\ }\gamma\in \Gamma. 
\end{cases}
\end{aligned}
\end{equation*}


We need the fact that $\gamma\cdot H''$'s are disjoint from each other to define $G_n$. Since $F_n$ restricts to identity on $\partial H''$, $G_n$ restricts to the identity on $\gamma\cdot \partial H''$, so $G_n$ is well-defined. By the definition of $G_n$, it is routine to check that it is a $\Gamma$-equivariant map, i.e. $G_n(\gamma\cdot x)=\gamma\cdot G_n(x)$ for any $x\in M$ and $\gamma\in \Gamma$. Since $F_n$ is a piecewise smooth $(Cn)$-Lipschitz map and $\Gamma$ acts on $\tilde{M}$ by isometries, $G_n$ is also a $(Cn)$-Lipschitz map.

By Theorem \ref{technical} (1) (3), for any generic point $x_0 \in \tilde{M}\setminus \bigcup_{\gamma\in \Gamma}\gamma \cdot H''$, we have 
$$F_n^{-1}(x_0)=\{x_0,x_1,\cdots,  x_{(2n+1)^2}\},$$ 
such that $x_i\in H''$ for any $i=1,\cdots, (2n+1)^2$. Moreover, the local degree of $F_n$ at $x_0$ is $1$, and the local degree of $F_n$ at $x_i$ is $-1$ for any $i=1,\cdots, (2n+1)^2$. Then for any generic $x_0\in \tilde{M}\setminus \bigcup_{\gamma\in \Gamma}\gamma \cdot H''$, its $\Gamma$-orbit is contained in $\tilde{M}\setminus \bigcup_{\gamma\in \Gamma}\gamma \cdot H''$ and
$$G_n^{-1}(x_0)=\{x_0\}\cup \Big(\cup_{\gamma\in \Gamma}\gamma\cdot (F_n^{-1}(\gamma^{-1}x_0)\cap H''))\Big).$$
The local degree of $G_n$ at $x_0$ is $1$. Each $\gamma\cdot (F_n^{-1}(\gamma^{-1}x_0)\cap H'')$ has cardinality $(2n+1)^2$, and the local degree of $G_n$ at each such point is $-1$. So we have 
$$\text{deg}(G_n)=1-(2n+1)^2\cdot |\Gamma|.$$

Since $G_n:\tilde{M} \to \tilde{M}$ is $\Gamma$-equivariant, it decends to a map $K_n:M\to M$, such that $\text{deg}(K_n)=\text{deg}(G_n)=1-(2n+1)^2\cdot |\Gamma|$. Since the metric $\tilde{g}$ on $\tilde{M}$ is the pull-back of the metric $g$ on $M$, $K_n$ is also a piecewise smooth $(Cn)$-Lipschitz map.  These facts on the sequence of maps $K_n: M\to M$ imply that $\alpha(M)\geq 2$ holds, as in the proof of Lemma \ref{easy}.
\end{proof}

By Proposition \ref{upper bound}, Lemma \ref{easy}, and Theorem \ref{sphericalsummand}, we know that $\alpha(M)=2$ for all $M$ as in Theorem \ref{nongeometric}. So Theorem \ref{nongeometric} holds, which implies our main result Theorem \ref{non-geometric}.

\section{Construction of Lipschitz maps}\label{technicalsection}
In this section,  we will prove Theorem \ref{technical}.

At first, we observe that we only need to prove Theorem \ref{technical} for one Riemannian metric on $M_k$. 

\begin{lem}\label{onemetric}
If Theorem \ref{technical} holds for one Riemannian metric $g_0$ on $M_k$, then it holds for any Riemannian metric $g$ on $M_k$.
\end{lem}

This result is proved by applying the argument in Lemma \ref{homotopy} to $id:(M_k,g)\to (M_k,g_0)$ and $id:(M_k,g_0)\to (M_k,g)$. The proof is straight forward and we skip the proof.

\subsection{Equip a Riemannian metric to $M_k=\#^kS^2\times S^1$}\label{equip}

Since we need to delicately control Lipschitz constants, we want to fix a  concrete metric on $M_k=\#^kS^2\times S^1=H\cup H'$ with its standard genus-$k$ Heegaard splitting. We take $3k-3$ properly embedded disks in $H$ that divide it to $2k-2$ copies of 
the thickened $Y$-shaped $3$-manifold $P$. Figure \ref{Figure1} shows this decomposition for the $k=5$ case, and the $Y$-shaped $3$-manifold $P$ is shown on the left side of Figure \ref{Figure2}. For each $P$, there are three disks $D_1, D_2, D_3$ in $\partial P$, where $D_i$ corresponds to a disk in $H$
numbered by $i\in \{1, 2, 3\}$.  We identify each copy of $(P, D_1, D_2, D_3)$ with the following subset $(P(3), D_1, D_2, D_3)$ of the 3-dimensional Euclidean space defined below.

For any $r>0$, let $P(r)$ be the following closed subset of $\mathbb{R}^3$:
\begin{equation}\label{Yshape}
\begin{aligned}
P(r)=&\ \{(x,y,z)\ |\ x^2+y^2\leq r^2, 0\leq z\leq 8\}\cup \{(x,y,z)\ |\ x^2+(y+z)^2\leq r^2,-8\leq z\leq 0\}\\
 &\ \cup \{(x,y,z)\ |\ x^2+(y-z)^2\leq r^2,-8\leq z\leq 0\}\\
 = &\ \{(x,y,z)\ |\ x^2+y^2\leq r^2, 0\leq z\leq 8\}\cup \{(x,y,z)\ |\ x^2+(y+z)^2\leq r^2,-8\leq z\leq 0, y\geq 0\}\\
 &\ \cup \{(x,y,z)\ |\ x^2+(y-z)^2\leq r^2,-8\leq z\leq 0, y\leq 0\}.
 \end{aligned}
 \end{equation}

For $P(3)$, we have the following three horizontal discs on its boundary: 
$$D_1=\{(x,y,8)\ |\ x^2+y^2\leq 9\},\ D_2=\{(x,y,-8)\ |\ x^2+(y-8)^2\leq 9\},\ D_3=\{(x,y,-8)\ |\ x^2+(y+8)^2\leq 9\}.$$ We use $\partial_{ver}P(3)$ to denote $\partial P(3)\setminus \text{int}(D_1\cup D_2\cup D_3)$, which is the vertical boundary of $P(3)$.
In the second decomposition of $P(r)$ in \eqref{Yshape}, the three closed subsets, denoted by $Q_1$, $Q_2$ and $Q_3$, are disjoint in their interiors.

See the left side of Figure \ref{Figure2} for $P(3)$, $D_1$, $D_2$, $D_3$, and $\partial_{ver}P(3)$, and the right side of Figure \ref{Figure2} for $Q_1$, $Q_2$ and $Q_3$.
  
We take $2k-2$ copies of $P(3)$, and use the identity maps on $D_i, i=1,2,3$ to paste them together by the pattern given in Figure \ref{Figure1}, to obtain $H_k$ back. Then $H_k$ has the Riemannian metric (indeed a flat metric) induced from  the Euclidean metric on $P(3)$.
We extend this Riemannian metric on $H_k$ to a Riemannian metric $g_0$ on $M_k$, and we will prove Theorem
\ref{technical} for this metric.

\begin{figure}[htbp]
    \centering
    \def\svgwidth{5in} 
  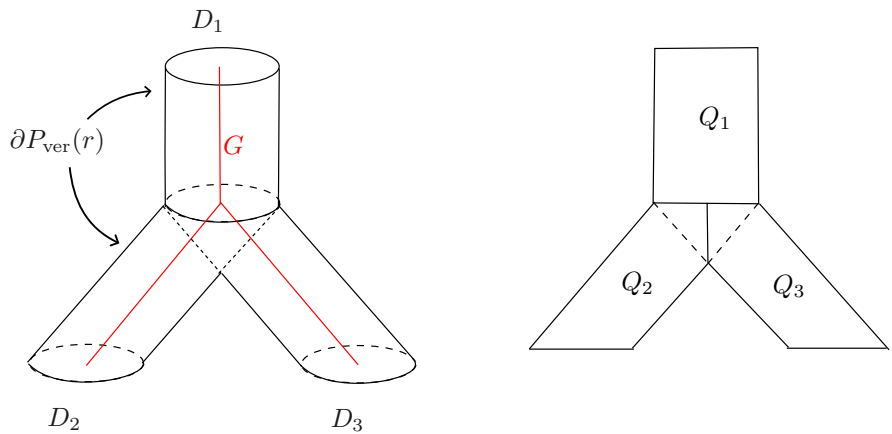
    \caption{A picture of $P(r)$.}
     \label{Figure2}
\end{figure}




The proof of Theorem \ref{technical} consists of a sequence of constructions, and most of our constructions will be done in $P(3)$.

\subsection{Graphs $G$ and $G'_{i,j}$ in the $Y$-shaped 3-manifold $P(3)$ and their packing}\label{graphs}

The intersection of all $P(r)$ with $r>0$ is a $Y$-shaped graph: 
\begin{equation}\label{straightgraph}
\begin{aligned}
G=\{(0,0,z)\ |\ 0\leq z\leq 8\} \cup \{(0,\epsilon z,z)\ |\ -8\leq z\leq 0, \epsilon=\pm 1\},
\end{aligned}
\end{equation}
as shown on the left side of Figure \ref{Figure2}.
For any $r>0$, $P(r)$ is the closed horizontal $r$-neighborhood $\mathcal{N}^{\text{hor}}_r(G)$ of $G$, i.e.
$$P(r)=\mathcal{N}^{\text{hor}}_r(G)=\{(x,y,z)\in \mathbb{R}^3\ |\ \exists(x_0,y_0,z)\in G \text{\ such\ that\ }(x-x_0)^2+(y-y_0)^2\leq r^2\}.$$

Now we fix a positive integer $n$. The following $Y$-shaped graph $G'$ in $P(3)$ is a small perturbation of $G$, and it will play a crucial role in the proof of Theorem \ref{technical}:
\begin{equation}\label{robot}
\begin{aligned}
G'= &\{(0,0,z)\ |\ 0\leq z\leq 8\} \cup \{(\epsilon z,0,z)\ |\ -\frac{1}{10n}\leq z\leq 0,\epsilon=\pm 1\}\\
\cup & \{(-\epsilon\frac{1}{10n},\epsilon(z+\frac{1}{10n}),z)\ |\ -8\leq z\leq -\frac{1}{10n},\epsilon=\pm1\}.
\end{aligned}
\end{equation}
Note that $G'\subset P(\frac{1}{5n})$ holds.

 For any pair of integers $i,j\in [-n,n]\cap \mathbb{Z}$, we define $G'_{i,j}$ to be the horizontal translation of $G'$ by $(\frac{i}{n},\frac{j}{n})$. In particular, $G'=G'_{0,0}$ holds. 
 
 The left side of Figure \ref{Figure3} shows how these  $G'_{i, j}$ stay along the $y$-axis, which look like robots walking on a line. The right side of Figure \ref{Figure3} provides a better view of $G'$.
 
 \begin{figure}[htbp]
    \centering
    \def\svgwidth{5in} 
\begingroup%
  \makeatletter%
  \providecommand\color[2][]{%
    \errmessage{(Inkscape) Color is used for the text in Inkscape, but the package 'color.sty' is not loaded}%
    \renewcommand\color[2][]{}%
  }%
  \providecommand\transparent[1]{%
    \errmessage{(Inkscape) Transparency is used (non-zero) for the text in Inkscape, but the package 'transparent.sty' is not loaded}%
    \renewcommand\transparent[1]{}%
  }%
  \providecommand\rotatebox[2]{#2}%
  \newcommand*\fsize{\dimexpr\f@size pt\relax}%
  \newcommand*\lineheight[1]{\fontsize{\fsize}{#1\fsize}\selectfont}%
  \ifx\svgwidth\undefined%
    \setlength{\unitlength}{595.27559055bp}%
    \ifx\svgscale\undefined%
      \relax%
    \else%
      \setlength{\unitlength}{\unitlength * \real{\svgscale}}%
    \fi%
  \else%
    \setlength{\unitlength}{\svgwidth}%
  \fi%
  \global\let\svgwidth\undefined%
  \global\let\svgscale\undefined%
  \makeatother%
  \begin{picture}(1,0.52380952)%
    \lineheight{1}%
    \setlength\tabcolsep{0pt}%
    \put(0,0){\includegraphics[width=\unitlength,page=1]{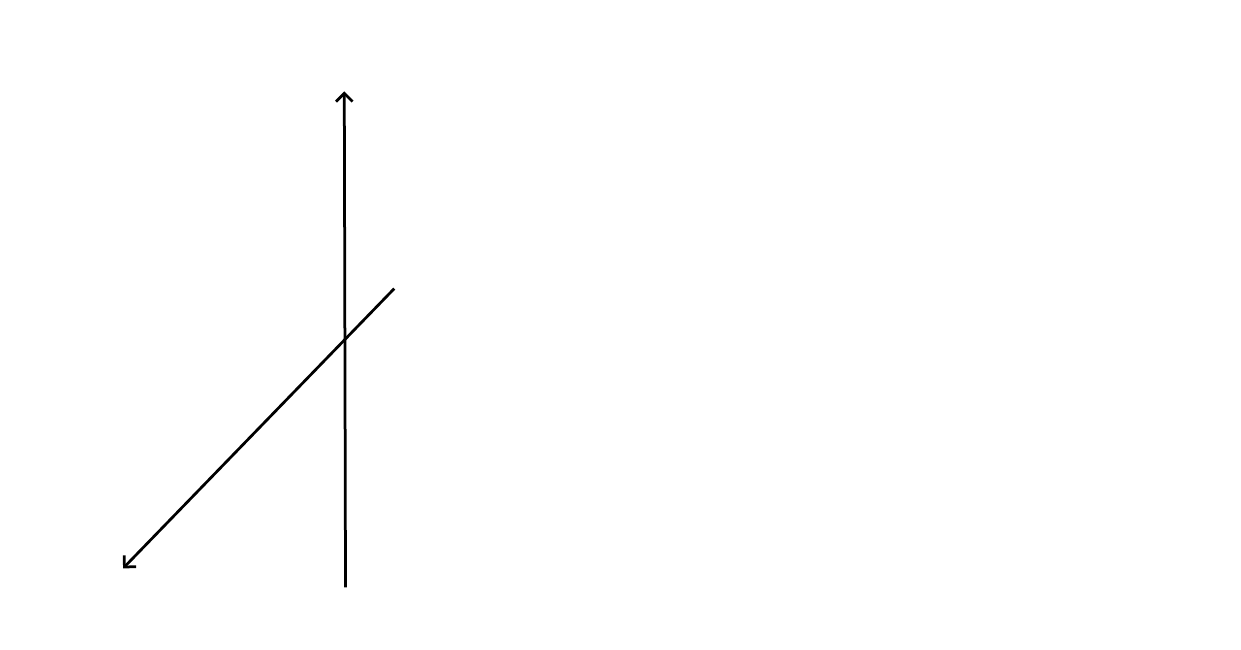}}%
    \put(0.08157585,0.03729033){\color[rgb]{0,0,0}\makebox(0,0)[lt]{\lineheight{1.25}\smash{\begin{tabular}[t]{l}$x$\end{tabular}}}}%
    \put(0.45932016,0.2778752){\color[rgb]{0,0,0}\makebox(0,0)[lt]{\lineheight{1.25}\smash{\begin{tabular}[t]{l}$y$\end{tabular}}}}%
    \put(0.25600051,0.4749797){\color[rgb]{0,0,0}\makebox(0,0)[lt]{\lineheight{1.25}\smash{\begin{tabular}[t]{l}$z$\end{tabular}}}}%
    \put(0,0){\includegraphics[width=\unitlength,page=2]{New_Figure_3.pdf}}%
    \put(0.14521965,0.38744171){\color[rgb]{0,0.50196078,0}\makebox(0,0)[lt]{\lineheight{1.25}\smash{\begin{tabular}[t]{l}$G_{0.-1}$\end{tabular}}}}%
    \put(0.22553963,0.34150617){\color[rgb]{0,0,1}\makebox(0,0)[lt]{\lineheight{1.25}\smash{\begin{tabular}[t]{l}$G_{0,0}$\end{tabular}}}}%
    \put(0.28605165,0.29861872){\color[rgb]{1,0,0}\makebox(0,0)[lt]{\lineheight{1.25}\smash{\begin{tabular}[t]{l}$G_{0,1}$\end{tabular}}}}%
    \put(0,0){\includegraphics[width=\unitlength,page=3]{New_Figure_3.pdf}}%
  \end{picture}%
\endgroup%

    \caption{$G'_{i,j}$ along the $y$-axis and another view of $G'$, with rescaled coordinate.}
     \label{Figure3}
\end{figure}
 
 The following packing property of $G'_{i,j}$ is important for us.

\begin{lem}\label{packing}
For any integers $i,j,i',j'\in [-n,n]\cap \mathbb{Z}$, if $(i,j)\ne (i',j')$, the following holds.
\begin{enumerate}
\item $\mathcal{N}^{hor}_{(100n)^{-1}}(G'_{i,j})$ is contained in $P(2)$.
\item $\mathcal{N}^{hor}_{(100n)^{-1}}(G'_{i,j})$ is disjoint from $\mathcal{N}^{hor}_{(100n)^{-1}}(G'_{i',j'})$. 
\end{enumerate}
\end{lem}

\begin{proof}
Since $G'\subset P(\frac{1}{5n})$ and $\|(\frac{i}{n},\frac{j}{n})\|\leq \sqrt{2}$, we have $G'_{i,j}\subset P(\sqrt{2}+\frac{1}{5n})$. Thus $$\mathcal{N}^{hor}_{(100n)^{-1}}(G'_{i,j})\subset P(\sqrt{2}+\frac{1}{5n}+\frac{1}{100n})\subset P(2).$$
This finishes the proof of (1), and we will prove (2) now.

We suppose that $\mathcal{N}^{hor}_{(100n)^{-1}}(G'_{i,j})\cap \mathcal{N}^{hor}_{(100n)^{-1}}(G'_{i',j'})\ne \emptyset$, then we take a point $(x_0,y_0,z_0)$ in the intersection. By the definition of horizontal neighborhood, there exists $(x_1,y_1,z_0)\in G'_{i,j}$ and $(x_2,y_2,z_0)\in G'_{i',j'}$ such that $\|(x_0-x_1,y_0-y_1)\|,\|(x_0-x_2,y_0-y_2)\|\leq \frac{1}{100n}$, so we have 
\begin{equation}\label{4}
\begin{aligned}
\|(x_1-x_2,y_1-y_2)\|\leq \frac{1}{50n}.
\end{aligned}
\end{equation}

We will get a contradiction by considering possible values of $z_0$, as in the definition of $G'$ in equation (\ref{robot}).
\begin{enumerate}
\item If $0\leq z_0\leq 8$, we have $(x_1,y_1)=(\frac{i}{n},\frac{j}{n})$ and $(x_2,y_2)=(\frac{i'}{n},\frac{j'}{n})$. Since $(i,j)\ne(i',j')$, we have $\|(x_1-x_2,y_1-y_2)\|\geq\frac{1}{n}$, which contradicts with equation (\ref{4}). 

\item If $-\frac{1}{10n}\leq z_0\leq 0$, we have $(x_1,y_1)=(\epsilon_1 z_0,0)+(\frac{i}{n},\frac{j}{n})$ and $(x_2,y_2)=(\epsilon_2 z_0,0)+(\frac{i'}{n},\frac{j'}{n})$, with $\epsilon_1,\epsilon_2\in \{1,-1\}$. Since $(i,j)\ne(i',j')$, we have $\|(x_1-x_2,y_1-y_2)\|\geq\|(\frac{i-i'}{n},\frac{j-j'}{n})\|-2|z_0|\geq \frac{1}{n}-\frac{1}{5n}=\frac{4}{5n}$, which contradicts with equation (\ref{4}).

\item If $-8\leq z_0\leq -\frac{1}{10n}$,  we have $(x_1,y_1)=(-\epsilon_1\frac{1}{10n},\epsilon_1(z_0+\frac{1}{10n}))+(\frac{i}{n},\frac{j}{n})$ and $(x_2,y_2)=(-\epsilon_2\frac{1}{10n},\epsilon_2(z_0+\frac{1}{10n}))+(\frac{i'}{n},\frac{j'}{n})$, with $\epsilon_1,\epsilon_2\in \{1,-1\}$.  If $i\ne i'$ or $\epsilon_1\ne \epsilon_2$, then we have $|x_1-x_2|\geq \frac{1}{5n}$, which contradicts with equation (\ref{4}). So we must have $i=i'$, $j\ne j'$, and $\epsilon_1=\epsilon_2$. In this case, we have $|y_1-y_2|\geq \frac{1}{n}$, which still contradicts with equation (\ref{4}).
\end{enumerate}

So the proof of (2) is done.
\end{proof}


Note that the disjointness in Lemma \ref{packing} fails if we replace $G'$ by $G$. This is the reason that we want to consider the graph $G'\subset P(3)$, which is less natural than $G$.

\subsection{Construct bi-Lipschitz homeomorphisms $(P(3),\mathcal{N}_n(G))\to (P(3),\mathcal{N}_n(G'_{i,j}))$}

For any positive integer $n$ and $i,j\in [-n,n]\cap \mathbb{Z}$, we will construct small tubular neighborhoods $\mathcal{N}_n(G)$ and $\mathcal{N}_n(G'_{i,j})$ of $G$ and $G'_{i,j}$, respectively, and will construct a bi-Lipschitz homeomophism $(P(3),\mathcal{N}_n(G),G)\to (P(3),\mathcal{N}_n(G'_{i,j}),G'_{i,j})$. This construction divides to two steps.


To construct homeomorphisms with controlled bi-Lipschitz constants, we will use time-$1$ maps of flows generated by vector fields. So we need the following theorem, which can be found in \cite[Theorem 2]{KSSV}. Here a vector field is said to be complete if it generates a flow $F: M\times \mathbb{R}\to M$ defined on the whole $\mathbb{R}$.

\begin{thm}[\cite{KSSV}]\label{Lipschitz}
Let $(M, g)$ be a Riemannian manifold, let $X$ be a complete vector field on $M$, and let $p,q\in M$. Suppose that $C=\sup_{p\in M}\|\nabla X(p)\|_g<\infty$. Then we have $$d(p(t),q(t))\leq d(p,q)\cdot e^{Ct}$$ for all $t\in [0,\infty)$, where $d(\cdot, \cdot)$ denotes the Riemannian distance on $M$.
\end{thm}

We will only apply Theorem \ref{Lipschitz} to the case that $M=P(3)\subset \mathbb{R}^3$, and $g$ is the standard Euclidean metric. In this case, Theorem \ref{Lipschitz} should have a more elementary proof, but we can not find it in the literature.

To apply Theorem \ref{Lipschitz}, we need to construct a piecewise smooth vector field $\vec{v}$ on $P(3)$. At first, we take a vector field $\vec{w}$ on $P(3)$ defined as:
\begin{equation}\label{vectorfield}
\begin{aligned}
\vec{w}(x,y,z)=
\begin{cases}
(x+y,-x-y,0)& \text{if\ } |x+y|\leq \frac{1}{10n},\\
(\frac{1}{10n}\text{sign}(x+y), -\frac{1}{10n}\text{sign}(x+y), 0) & \text{if\ }|x+y|\geq \frac{1}{10n}.
\end{cases}
\end{aligned}
\end{equation}
Here $\text{sign}(x+y)$ denotes the sign of $x+y$. Our $\vec{w}$ is a piecewise smooth vector field on $P(3)$, and it is tangent to its non-smooth locus $\{(x,y,z)\in P(3)\ |\ |x+y|=\frac{1}{10n}\}.$ It is easy to check that the $\|w\|\leq \frac{1}{5n}$ and $\|\nabla \vec{w}\|\leq 2$ hold away from the non-smooth locus. Note that $\vec{w}$ always has zero $z$-coordinate, and is indendent of $z$. See Figure \ref{Figure4} for the vector field $\vec{w}$ in the $xy$-plane.

Then we take a smooth function $\psi:P(3)\to [0,1]$ such that $\psi(x,y,z)=1$ for any $(x,y,z)\in P(2)$, and $\psi(x,y,z)=0$ for any $(x,y,z)\in P(3)\setminus P(\frac{5}{2})$. Since $P(3)$ is compact, there exists a constant $C_0>0$ such that $\|\nabla \psi\|\leq C_0$ holds.  The desired piecewise smooth vector field $\vec{v}$ is defined by 
$$\vec{v}(x,y,z)=\psi(x,y,z)\cdot \vec{w}(x,y,z).$$ 
Note that $\vec{v}$ vanishes on $\partial_{ver}P(3)$, is tangent to its non-smooth locus, and is tangent to all horizontal planes $z=c$. Moreover, $\nabla \vec{v}$ satisfies the following inequality away from its non-smooth locus:
\begin{equation}\label{6}
\begin{aligned}
\|\nabla \vec{v}\|\leq |\psi|\cdot \|\nabla \vec{w}\|+\|\vec{w}\|\cdot \|\nabla \psi \|\leq 1\cdot 2 +\frac{1}{5n}\cdot C_0\leq C_0+2.
\end{aligned}
\end{equation}

\begin{figure}[htbp]
    \centering
    \def\svgwidth{4in} 
\begingroup%
  \makeatletter%
  \providecommand\color[2][]{%
    \errmessage{(Inkscape) Color is used for the text in Inkscape, but the package 'color.sty' is not loaded}%
    \renewcommand\color[2][]{}%
  }%
  \providecommand\transparent[1]{%
    \errmessage{(Inkscape) Transparency is used (non-zero) for the text in Inkscape, but the package 'transparent.sty' is not loaded}%
    \renewcommand\transparent[1]{}%
  }%
  \providecommand\rotatebox[2]{#2}%
  \newcommand*\fsize{\dimexpr\f@size pt\relax}%
  \newcommand*\lineheight[1]{\fontsize{\fsize}{#1\fsize}\selectfont}%
  \ifx\svgwidth\undefined%
    \setlength{\unitlength}{595.27559055bp}%
    \ifx\svgscale\undefined%
      \relax%
    \else%
      \setlength{\unitlength}{\unitlength * \real{\svgscale}}%
    \fi%
  \else%
    \setlength{\unitlength}{\svgwidth}%
  \fi%
  \global\let\svgwidth\undefined%
  \global\let\svgscale\undefined%
  \makeatother%
  \begin{picture}(1,0.57142857)%
    \lineheight{1}%
    \setlength\tabcolsep{0pt}%
    \put(0,0){\includegraphics[width=\unitlength,page=1]{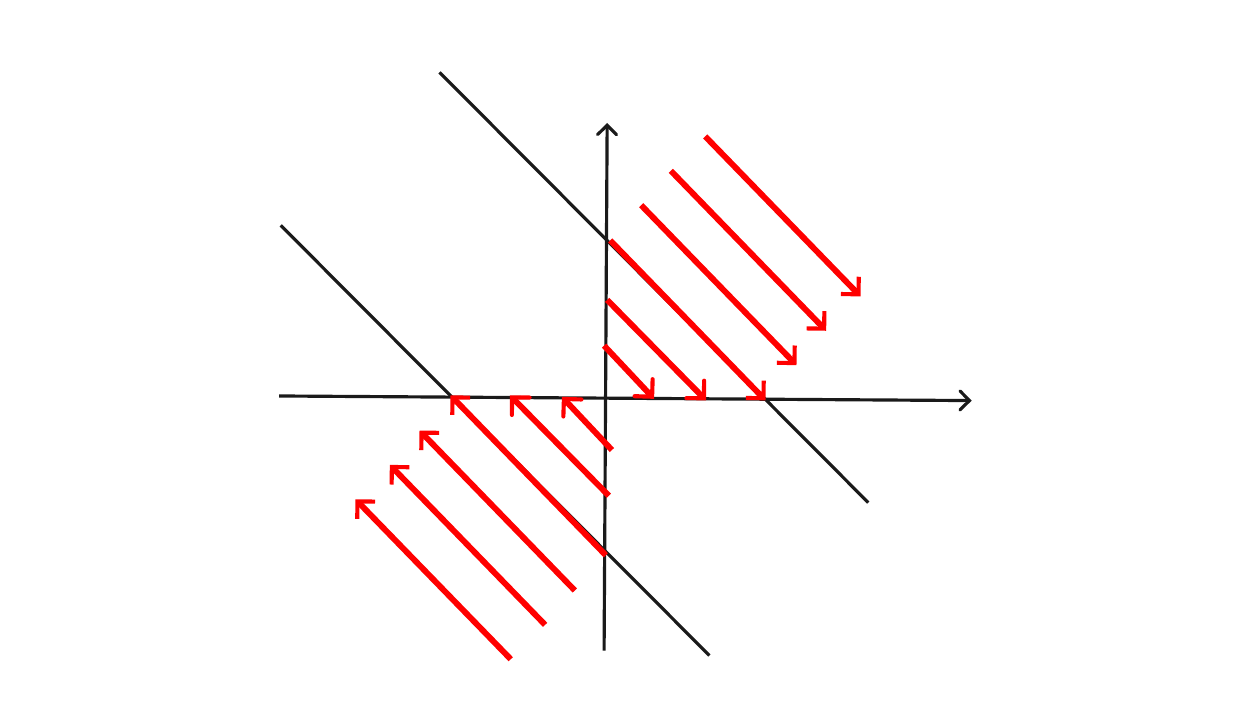}}%
    \put(0.79649164,0.24106275){\color[rgb]{0.10196078,0.10196078,0.10196078}\makebox(0,0)[lt]{\lineheight{1.25}\smash{\begin{tabular}[t]{l}$x$\end{tabular}}}}%
    \put(0.47234882,0.50183871){\color[rgb]{0.10196078,0.10196078,0.10196078}\makebox(0,0)[lt]{\lineheight{1.25}\smash{\begin{tabular}[t]{l}$y$\end{tabular}}}}%
    \put(0.67717438,0.37760113){\color[rgb]{0.10196078,0.10196078,0.10196078}\makebox(0,0)[lt]{\lineheight{1.25}\smash{\begin{tabular}[t]{l}$x+y\geq\frac{1}{10n}$\end{tabular}}}}%
    \put(0.09911166,0.10575445){\color[rgb]{0.10196078,0.10196078,0.10196078}\makebox(0,0)[lt]{\lineheight{1.25}\smash{\begin{tabular}[t]{l}$x+y\leq-\frac{1}{10n}$\end{tabular}}}}%
    \put(0.50274248,0.15219883){\color[rgb]{0.10196078,0.10196078,0.10196078}\makebox(0,0)[lt]{\lineheight{1.25}\smash{\begin{tabular}[t]{l}$|x+y|\leq \frac{1}{10n}$\end{tabular}}}}%
  \end{picture}%
\endgroup%

    \caption{Vector field $\vec{w}$ in the $xy$ plane.}
     \label{Figure4}
\end{figure}
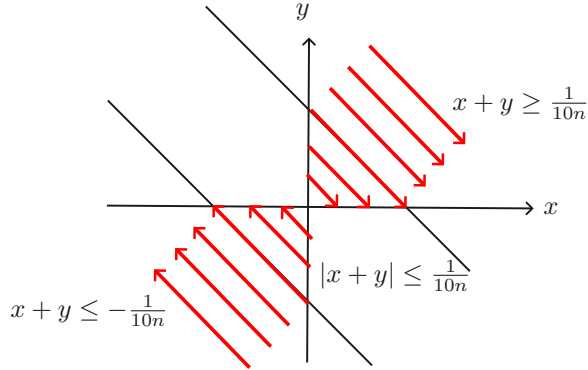

Since $\vec{v}$ vanishes on $\partial_{ver}P(3)$ and is tangent to $D_1, D_2, D_3$, $\vec{v}$ generates a piecewise smooth flow $\Phi: P(3)\times \mathbb{R}\to P(3)$. We use $f_1:P(3)\to P(3)$ to denote the time-$1$ map of $\Phi$, i.e. $f_1(p)=\Phi(p,1)$. We want to check the following properties of $f_1$.

\begin{lem}\label{elementary}
The piecewise smooth map $f_1:P(3)\to P(3)$ satisfies the following properties.
\begin{enumerate}
\item The restriction of $f_1$ to $\partial_{ver}P(3)$ is the identity.
\item $f_1$ is a $C_1$-bi-Lipschitz homeomorphism, where $C_1=e^{C_0+2}>1$ does not depend on $n$.
\item $f_1(G)=G'$.
\item $f_1(P(\frac{1}{100C_1n}))\subset \mathcal{N}^{hor}_{(100n)^{-1}}(G')$.
\end{enumerate}
\end{lem}


\begin{proof}
Since the vector field $\vec{v}$ vanishes on $\partial_{ver}P(3)$, $\Phi(x,t)=x$ holds for any $x\in \partial_{ver}P(3)$. So $f_1$ restricts to the identity on $\partial_{ver}P(3)$, thus (1) holds.

Note that the vector field $\vec{v}$ is smooth on 
$$Q=P(3)\cap \{(x,y,z)\ |\ |x+y|\leq \frac{1}{10n}\}\ \text{and}\ Q'=P(3)\cap \{(x,y,z)\ |\ |x+y|\geq \frac{1}{10n}\},$$ 
and is tangent to its non-smooth locus $Q\cap Q'$. So both the flow $\Phi$ and the map $f_1$ preserve $Q$ and $Q'$. Since $\|\nabla \vec{v}\|\leq C_0+2$ away from the non-smooth locus, Theorem \ref{Lipschitz} implies that the restrictions of $f_1$ to $Q$ and $Q'$ are both $C_1$-Lipschitz for $C_1=e^{C_0+2}$, thus $f_1$ is a $C_1$-Lipschitz map. Since the time-$1$ map of $-\vec{v}$ is $f_1^{-1}$, the same argument implies that $f_1^{-1}$ is $C_1$-Lipschitz. So $f_1$ is a $C_1$-bi-Lipschitz homeomorphism, thus (2) holds.

Now we work on item (3), which is checked by solving elementary ODEs.  For any $(x_0,y_0,z_0)\in G$, we want to solve the integral curve $\gamma(t)=(x(t),y(t),z(t))$ of $\vec{v}$ such that $\gamma(0)=(x_0,y_0,z_0)$. If $0\leq z_0\leq 8$, then $(x_0,y_0,z_0)=(0,0,z_0)$ and it is a zero point of $\vec{v}$. So $\gamma(1)=(0,0,z_0)\in G'$ holds. If $-\frac{1}{10n}\leq z_0\leq 0$, then $(x_0,y_0,z_0)=(0,\epsilon z_0,z_0)$. The integral curve is given by $\gamma(t)=(\epsilon z_0t,\epsilon z_0(1-t),z_0)$, and $\gamma(1)=(\epsilon z_0,0,z_0)\in G'$. Note that $\gamma(t)$ always lies in $Q\cap P(2)$ in this case. If $-8\leq z_0\leq -\frac{1}{10n}$, then $(x_0,y_0,z_0)=(0,\epsilon z_0,z_0)$. The integral curve is given by $\gamma(t)=(-\frac{1}{10n}\epsilon t,\epsilon (z_0+\frac{1}{10n}t),z_0)$, and $\gamma(1)=(-\frac{1}{10n}\epsilon,\epsilon (z_0+\frac{1}{10n}),z_0)\in G'$. Note that $\gamma(t)$ still lies in $Q'\cap P(2)$ in this case. The proof of (3) is done.

Note that $f_1$ preserves the $z$-coordinate of each point. So (2) and (3) implies that $f_1$ maps $P(\frac{1}{100C_1n})=\mathcal{N}^{hor}_{(100C_1n)^{-1}}(G)$ into $\mathcal{N}^{hor}_{(100n)^{-1}}(G')$, thus (4) holds.
\end{proof}

Now we take $\mathcal{N}_n(G)=P(\frac{1}{100C_1n})=\mathcal{N}^{hor}_{(100C_1n)^{-1}}(G)$ to be our desired small neighborhood of $G$, and
we take $\mathcal{N}_n(G')=f_1(P(\frac{1}{100C_1n}))$ to be our desired small neighborhood of $G'$. Note that we do not have an explicit formula that defines $\mathcal{N}_n(G')$. Then Lemma \ref{elementary} implies that $f_1:(P(3),\mathcal{N}_n(G),G)\to (P(3),\mathcal{N}_n(G'),G')$ is a piecewise smooth $C_1$-bi-Lipschitz homeomorphism, and $\mathcal{N}_n(G')=f_1(P(\frac{1}{100C_1n})) \subset \mathcal{N}^{hor}_{(100n)^{-1}}(G')$ holds.

For any $i,j\in [-n,n]\cap \mathbb{Z}$, let $\mathcal{N}_n(G'_{i,j})$ be the horizontal translation of $\mathcal{N}_n(G')$ by $(\frac{i}{n},\frac{j}{n})$. Then Lemma \ref{elementary} (4) implies $\mathcal{N}_n(G'_{i,j})\subset \mathcal{N}^{hor}_{(100n)^{-1}}(G'_{i,j})$, and Lemma \ref{packing} (2) implies that $\mathcal{N}_n(G'_{i,j})$ and $\mathcal{N}_n(G'_{i',j'})$ are disjoint from each other if $(i,j)\ne (i',j')$.

Now we prove that $(P(3),\mathcal{N}_n(G'),G')$ and $(P(3),\mathcal{N}_n(G'_{i,j}),G'_{i,j})$ are bi-Lipschitz homeomorphic to each other.

\begin{lem}\label{translation}
There exists a constant $C_2>1$ indepent of $n$, such that for any $i,j\in [-n,n]\cap \mathbb{Z}$, there exists a smooth $C_2$-bi-Lipschitz homeomorphism $f_{2,i,j}:(P(3),\mathcal{N}_n(G'),G')\to (P(3),\mathcal{N}_n(G'_{i,j}),G'_{i,j})$ that restricts to the identity on $\partial_{ver}P(3)$.
\end{lem} 

\begin{proof}
Let $\vec{w}'(x,y,z)=(\frac{i}{n},\frac{j}{n}, 0)$ be a constant horizontal vector field on $P(3)$, and let $\psi:P(3)\to [0,1]$ be the function we took before Lemma \ref{elementary}. Then we take a vector field $\vec{v}'$ on $P(3)$ by $\vec{v}'(x,y,z)=\psi(x,y,z)\cdot \vec{w}'(x,y,z)$, and we have 
$$\|\nabla \vec{v}'\|\leq |\psi|\cdot \|\nabla \vec{w}'\|+\|\vec{w}'\|\cdot \|\nabla \psi\|\leq 1\cdot 0+\sqrt{2}\cdot C_0=\sqrt{2}C_0.$$
As the vector field $\vec{v}$, $\vec{v}'$ is also tangent to $D_1,D_2,D_3$ and vanishes on $\partial_{vec}P(3)$.

Let $f_{2,i,j}:P(3)\to P(3)$ be the time-$1$ map generated by the vector field $\vec{v}'$. Theorem \ref{Lipschitz} implies that $f_{2,i,j}$ is $C_2$-Lipschitz for $C_2=e^{\sqrt{2}C_0}>1$, and the time-$(-1)$ map is also $C_2$-Lipschitz. So $f_{2,i,j}$ is a $C_2$-bi-Lipschitz homeomorphism, and it restricts to identity $\partial_{ver}P(3)$.

Note that $\vec{v}'$ restricts to the constant vector field $(\frac{i}{n},\frac{j}{n}, 0)$ in $P(2)$. For any $t\in [0,1]$, it is easy to check that the $(\frac{i}{n}t,\frac{j}{n}t)$-horizontal translation of $\mathcal{N}_n(G')$ is contained in $P(2)$, as the proof of Lemma \ref{packing} (1). So the restriction of $f_{2,i,j}$ to $\mathcal{N}_n(G')$ is simply the horizontal translation by $(\frac{i}{n},\frac{j}{n})$. In particular, it maps $G'$ to $G'_{i,j}$, and maps $\mathcal{N}_n(G')$ to $\mathcal{N}_n(G'_{i,j})$.
\end{proof}

\subsection{Construct a bi-Lipschitz homeomorphism $(P(3),\mathcal{N}_n(G))\to (P(3),P(2))$.}

In this subsection, we contruct a piecewise smooth bi-Lipschitz homeomoprhism $f_3:(P(3),\mathcal{N}_n(G),G)\to (P(3),P(2),G)$. This is the only homeomorphism in our construction whose Lipschitz number depends on $n$.

\begin{lem}\label{dependonn}
There exists a constant $C_3>1$, such that there exists a piecewise smooth homeomorphism $f_3:(P(3),\mathcal{N}_n(G),G)\to (P(3),P(2),G)$ and the following hold.
\begin{enumerate}
\item $f_3$ is a $(C_3n)$-Lipschitz map.
\item $f_3^{-1}$ is a $C_3$-Lipschitz map. 
\item The restriction of $f_3$ to $\partial_{vec}P(3)$ is the identity.
 \end{enumerate}
\end{lem}

\begin{proof}
Recall that in the expression of $P(3)$ in (\ref{Yshape}), the second term is a decomposition of $P(3)$ as three closed subsets that are disjoint in their interiors. We denote these three subsets by: 
\begin{equation*}
\begin{aligned}
& Q_1=\{(x,y,z)\ |\ x^2+y^2\leq 9, 0\leq z\leq 8\},\\
& Q_2=\{(x,y,z)\ |\ x^2+(y+z)^2\leq 9^2,-8\leq z\leq 0, y\geq 0\}, \\
& Q_3=\{(x,y,z)\ |\ x^2+(y-z)^2\leq 9^2,-8\leq z\leq 0, y\leq 0\}.
\end{aligned}
\end{equation*}
See the right side of Figure \ref{Figure2} for $Q_1, Q_2, Q_3$. Each $Q_i$ has a singular foliation given by $Q_i\cap \partial_{ver}P(r)$ for $r\in (0,3]$. 

For each $i=2,3$, we first construct a $C_3'$-bi-Lipshitz homeomorphism $g_i: Q_i\to Q_1$, such that $g_i$ sends $Q_i\cap \partial_{ver}P(r)$ to $Q_1\cap \partial_{ver}P(r)$ for any $r\in (0,3]$. The map $g_2$ is defined by 
\begin{equation*}
\begin{aligned}
g_2(x,y,z)=
\begin{cases}
(x,y+z,-z) & \text{if\ }y+z\geq 0,\\
(x,y+z,\frac{8y}{8+y+z}) & \text{if\ } y+z\leq 0.
\end{cases}
\end{aligned}
\end{equation*} 
$g_2$ is actually a composition of two maps. We first use the linear map $(x,y,z)\to (x,y+z,-z)$ to map $Q_2$ to a subset of $Q_1$. In particular, it maps the $y+z\geq 0$ part of $Q_2$ to the $y\geq 0$ part of $Q_1$ by homeomorphism, and maps the $y+z\leq 0$ part of $Q_2$ to the $y\leq 0, y+z\geq 0$ part of $Q_1$ by homeomorphism. Then we use the homeomorphism $(x,y,z)\to (x,y,\frac{8(z+y)}{8+y})$ to map the $y\leq 0, y+z\geq 0$ part of $Q_1$ to the $y\leq 0$ part of $Q_1$. 

It is easy to check that $g_2$ sends $Q_2\cap \partial_{ver}P(r)$ to  $Q_1\cap \partial_{ver}P(r)$ for all $r\in (0,3]$. The restriction of $g_2$ to $Q_2\cap Q_1=\{(x,y,0)\ |\ x^2+y^2\leq 9, y\geq 0\}$ is the identity, and we have $g_2(Q_2\cap Q_3)=Q_3\cap Q_1$. Since $g_2$ is a piecewise smooth homeomorhism between compact spaces, it is a $C_3'$-bi-Lipschitz homeomorphism for some constant $C_3'>1$. 

Let $\iota(x,y,z)=(x,-y,z)$ be the reflection along the $y$-plane. Then it preserves $Q_1$, and swaps $Q_2$ and $Q_3$. We define $g_3:Q_3\to Q_1$ by $g_3=\iota\circ g_2\circ \iota$. Since $\iota$ is an isometry, $g_3$ is also a $C_3'$-bi-Lipschitz homeomorphism. Moreover, $g_3$ sends $Q_3\cap \partial_{ver}P(r)$ to  $Q_1\cap \partial_{ver}P(r)$ for all $r\in (0,3]$, its restriction to $Q_3\cap Q_1=\{(x,y,0)\ |\ x^2+y^2\leq 9, y\leq 0\}$ is the identity, $g_3=\iota \circ g_2$ on $Q_2\cap Q_3=\{(x,0,z)\ |\ x^2+z^2\leq 9, z\leq 0\}$, and $g_3(Q_2\cap Q_3)=Q_2\cap Q_1$.

{\bf Claim:} There exists a constant $C_3''>1$, such that there is a piecewise smooth homeomorphsm $g: (Q_1,Q_1\cap \mathcal{N}_n(G),Q_1\cap G)\to  (Q_1,Q_1\cap P(2), Q_1\cap G)$ and the following hold.
\begin{enumerate}[(i)]
\item $g$ is a $(C_3''n)$-Lipschitz map.
\item $g^{-1}$ is a $C_3''$-Lipschitz map.
\item $g\circ \iota=\iota \circ g$.
\item $g=id$ on $Q_1\cap \partial_{ver}P(3)$.
\item $g$ preserves $Q_1\cap Q_2$ and $Q_1\cap Q_3$. 
\end{enumerate}

Once we have this map $g$, we define the desired homeomorphism $f_3:(P(3),\mathcal{N}_n(G),G)\to (P(3),P(2),G)$ by 
\begin{equation*}
\begin{aligned}
f_3(x,y,z)=
\begin{cases}
g(x,y,z)& \text{if\ }(x,y,z)\in Q_1\\
g_2^{-1}\circ g\circ g_2(x,y,z)& \text{if\ }(x,y,z)\in Q_2\\
g_3^{-1}\circ g\circ g_3(x,y,z)& \text{if\ }(x,y,z)\in Q_3.
\end{cases}
\end{aligned}
\end{equation*}
This map $f_3$ is well-defined since $g_2=id$ on $Q_1\cap Q_2$, $g_3=id$ on $Q_1\cap Q_3$, $g$ preserves $Q_1\cap Q_2$ and $Q_1\cap Q_3$, $g_3=\iota\circ g_2$ on $Q_2\cap Q_3$, and $g\circ \iota=\iota \circ g$. $f_3$ sends $\mathcal{N}_n(G)$ to $P(2)$ since $g$ sends $Q_1\cap \mathcal{N}_n(G)$ to $Q_1\cap P(2)$, and $g_2,g_3$ both preserve the intersectiosn with $\partial_{ver}P(r)$ for all $r\in(0,3]$. $f_3$ restricts to the identity on $\partial_{ver}P(3)$ since $g$ restricts to the identity on $Q_1\cap \partial_{ver}P(3)$. By the bi-Lipschitz constants of $g,g_2,g_3$, we know that $f_3$ is a $(C_3''(C_3')^2n)$-Lipschitz homeomorphism and $f_3^{-1}$ is $(C_3''(C_3')^2)$-Lipschitz, and we take the constant $C_3=C_3''(C_3')^2$ in this lemma.


Now it remains to prove the claim.

 Recall that $\mathcal{N}_n(G)=P(\frac{1}{100C_1n})$, and we will use the polar coordinate $(r,\theta,z)$ on $Q_1$. We define $g:(Q_1,Q_1\cap \mathcal{N}_n(G),Q_1\cap G)\to  (Q_1,Q_1\cap P(2), Q_1\cap G)$ by
\begin{equation*}
\begin{aligned}
g(r,\theta,z)=
\begin{cases}
(200C_1n r,\theta,z)& \text{if\ }r\leq \frac{1}{100C_1n},\\
(\frac{100C_1n}{300C_1n-1}r+3\frac{200C_1n-1}{300C_1n-1},\theta,z) & \text{if\ } \frac{1}{100C_1n}\leq r\leq 3.
\end{cases}
\end{aligned}
\end{equation*}
Then $g^{-1}$ is given by 
\begin{equation*}
\begin{aligned}
g^{-1}(r,\theta,z)=
\begin{cases}
(\frac{1}{200C_1n} r,\theta,z)& \text{if\ }r\leq 2,\\
((3-\frac{1}{100C_1n})r-3(2-\frac{1}{100C_1n}),\theta,z) & \text{if\ } 2\leq r\leq 3.
\end{cases}
\end{aligned}
\end{equation*}
Since $g$ preserves the $\theta$-coordinate, we have $g\circ \iota=\iota\circ g$, and $g$ preserves $Q_1\cap Q_2$ and $Q_1\cap Q_3$.  So items (iii) and (v) hold. Since $g(3,\theta,z)=(3,\theta,z)$, it restricts to the identity on $Q_1\cap \partial_{ver}P(3)$, thus (iv) holds. Now it remains to compute the Lipschitz constants of $g$ and $g^{-1}$.

We have four formulas in the definitions of $g$ and $g^{-1}$, and all of them are in the form of $h(r,\theta,z)= (f(r),\theta,z)$. At $(r_0,\theta_0,z_0)$, the tangent map $dh_{(r_0,\theta_0,z_0)}:T_{(r_0,\theta_0,z_0)}\mathbb{R}^3\to T_{(f(r_0),\theta_0,z_0)} \mathbb{R}^3$ is given by 
$$\frac{\partial }{\partial r}|_{(r_0,\theta_0,z_0)}\to f'(r_0)\frac{\partial }{\partial r}|_{(f(r_0),\theta_0,z_0)},\ \frac{\partial }{\partial \theta}|_{(r_0,\theta_0,z_0)}\to \frac{\partial }{\partial \theta}|_{(f(r_0),\theta_0,z_0)},\ \frac{\partial }{\partial z}|_{(r_0,\theta_0,z_0)}\to \frac{\partial }{\partial z}|_{(f(r_0),\theta_0,z_0)}.$$ 
Since $$\|\frac{\partial }{\partial r}|_{(r_0,\theta_0,z_0)}\|=1, \|\frac{\partial }{\partial \theta}|_{(r_0,\theta_0,z_0)}\|=r_0, \|\frac{\partial }{\partial z}|_{(r_0,\theta_0,z_0)}\|=1,$$ and these vectors form an orthogonal basis of $T_{(r_0,\theta_0,z_0)}$, we have $\|dh_{(r_0,\theta_0,z_0)}\|=\max{\{f'(r_0),\frac{f(r_0)}{r_0}\}}$.

We apply this computation to the formulas in $g$ and $g^{-1}$. The Lipschitz constant of $g$ equals
the maximum of the following numbers:
$$200C_1n, \frac{100C_1n}{300C_1n-1},\max\{\frac{100C_1n}{300C_1n-1}+\frac{3}{r}\frac{200C_1n-1}{300C_1n-1}\ |\ \frac{1}{100C_1n}\leq r \leq 3\},$$
and the Lipschitz constant of $g^{-1}$ equals the maximum of:
$$\frac{1}{200C_1n}, 3-\frac{1}{100C_1n}, \max \{(3-\frac{1}{100C_1n})-\frac{3}{r}(2-\frac{1}{100C_1n})\ |\ 2\leq r\leq 3\}.$$
A dirct computation shows that the Lipschitz constant of $g$ is given by the first and third terms, which is $200C_1n$; and the Lipschitz constant of $g^{-1}$ is given by the second term, which is at most $3$. So we can take the constant $C_3''$ to be $\max\{200C_1,3\}$.

The proofs of the claim and this lemma are done. A picture of the restriction of $f_3$ on the intersection of $P(3)$ and the $yz$-plane can be found in Figure \ref{Figure5}.

\begin{figure}[htbp]
    \centering
    \def\svgwidth{5.5in} 
\begingroup%
  \makeatletter%
  \providecommand\color[2][]{%
    \errmessage{(Inkscape) Color is used for the text in Inkscape, but the package 'color.sty' is not loaded}%
    \renewcommand\color[2][]{}%
  }%
  \providecommand\transparent[1]{%
    \errmessage{(Inkscape) Transparency is used (non-zero) for the text in Inkscape, but the package 'transparent.sty' is not loaded}%
    \renewcommand\transparent[1]{}%
  }%
  \providecommand\rotatebox[2]{#2}%
  \newcommand*\fsize{\dimexpr\f@size pt\relax}%
  \newcommand*\lineheight[1]{\fontsize{\fsize}{#1\fsize}\selectfont}%
  \ifx\svgwidth\undefined%
    \setlength{\unitlength}{595.27559055bp}%
    \ifx\svgscale\undefined%
      \relax%
    \else%
      \setlength{\unitlength}{\unitlength * \real{\svgscale}}%
    \fi%
  \else%
    \setlength{\unitlength}{\svgwidth}%
  \fi%
  \global\let\svgwidth\undefined%
  \global\let\svgscale\undefined%
  \makeatother%
  \begin{picture}(1,0.47619048)%
    \lineheight{1}%
    \setlength\tabcolsep{0pt}%
    \put(0,0){\includegraphics[width=\unitlength,page=1]{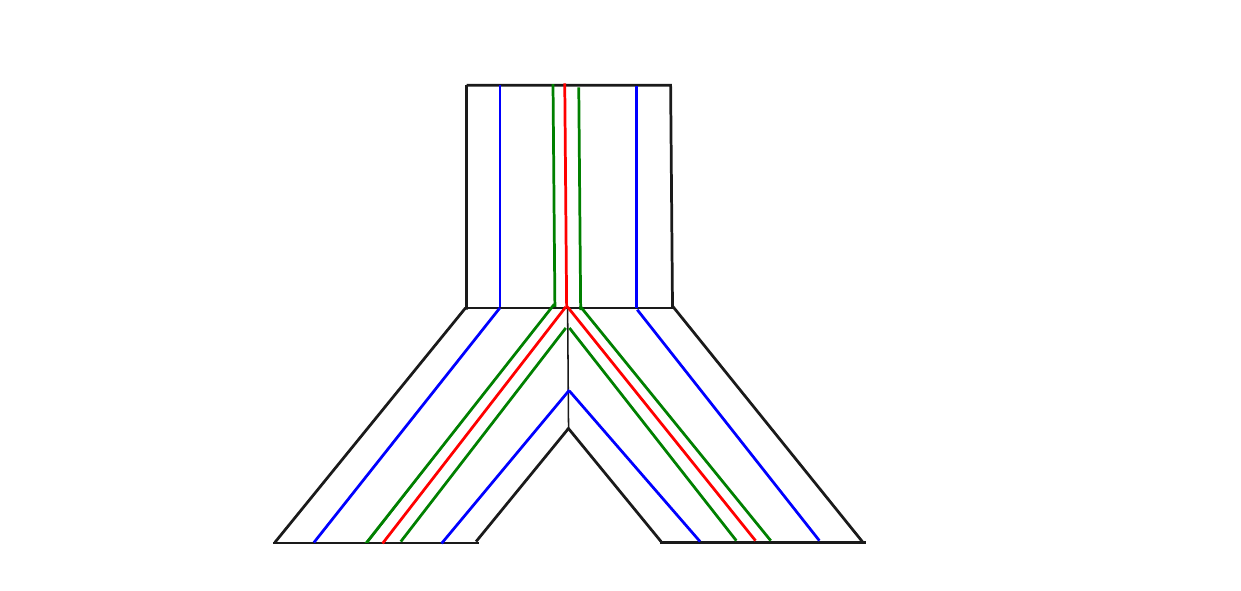}}%
    \put(0.28773815,0.25810145){\color[rgb]{0.10196078,0.10196078,0.10196078}\makebox(0,0)[lt]{\lineheight{1.25}\smash{\begin{tabular}[t]{l}$P(3)$\end{tabular}}}}%
    \put(0.40487033,0.42006262){\color[rgb]{0,0,1}\makebox(0,0)[lt]{\lineheight{1.25}\smash{\begin{tabular}[t]{l}$P(2)$\end{tabular}}}}%
    \put(0.6030463,0.00553735){\color[rgb]{1,0,0}\makebox(0,0)[lt]{\lineheight{1.25}\smash{\begin{tabular}[t]{l}$G$\end{tabular}}}}%
    \put(0.28521831,0.00185704){\color[rgb]{0,0.50196078,0}\makebox(0,0)[lt]{\lineheight{1.25}\smash{\begin{tabular}[t]{l}$\mathcal{N}_n(G)$\end{tabular}}}}%
    \put(0,0){\includegraphics[width=\unitlength,page=2]{New_Figure_5.pdf}}%
  \end{picture}%
\endgroup%

    \caption{A picture of the restriction of $f_3$ on $P(3)\cap yz$-plane, which sends $\mathcal{N}_n(G)$ to $P(2)$.}
     \label{Figure5}
\end{figure}
\end{proof}

\subsection{Proof of Theorem \ref{technical}}

Let's summarize the maps we constructed in the previous a few sections. The main object is the $Y$-shaped $3$-manifold $P(3)$ and $\partial P(3)=\partial_{ver}P(3)\cup (D_1\cup D_2\cup D_3)$, as in Figure \ref{Figure2}. In $P(3)$, we have a closed submanifold $P(2)\subset P(3)$. For any positive integer $n$, we have further closed submanifolds $\mathcal{N}_n(G),\mathcal{N}_n(G')\subset P(2)$. For any $i,j\in [-n,n]\cap \mathbb{Z}$, the $(\frac{i}{n},\frac{j}{n},0)$-translation of $\mathcal{N}_n(G')$ is denoted by $\mathcal{N}_n(G'_{i,j})$, and these manifolds are disjoint from each other in $P(2)$ (Lemma \ref{packing} (2)). All these submanifolds of $P(3)$ intersect with $\partial P(3)$ in the interior of $D_1\cup D_2\cup D_3$.

We have constructed the following piecewise smooth bi-Lipschitz homeomorphisms of $P(3)$ that restrict to the identity on $\partial_{ver}P(3)$. Recall that the constants $C_1,C_2,C_3$ below are independent of $n$
\begin{enumerate}
\item In Lemma \ref{elementary}, we constructed a $C_1$-bi-Lipschitz homeomorphism $f_1:(P(3),\mathcal{N}_n(G))\to (P(3),\mathcal{N}_n(G'))$.
\item In Lemma \ref{translation}, for any $i,j\in [-n,n]\cap \mathbb{Z}$, we constructed a $C_2$-bi-Lipschitz homeomorphism $f_{2,i,j}:(P(3),\mathcal{N}_n(G'))\to (P(3),\mathcal{N}_n(G'_{i,j}))$.
\item In Lemma \ref{dependonn}, we constructed a $(C_3n)$-Lipschitz homeomorphism $f_3:(P(3),\mathcal{N}_n(G))\to (P(3),P(2))$ such that $f_3^{-1}$ is $C_3$-Lipschitz.
\end{enumerate}

By Lemma \ref{onemetric}, we only need to prove Theorem \ref{technical} for one Riemannian metric $g_0$ on $M_k=H\cup H'$. We take $g_0$ such that its restriction to $H$ is isometric to $H_k$ (as defined in Section \ref{equip}). Recall that the metric on $H_k$ is obtained by identifying $2k-2$ copies of $P(3)$ (with the Euclidean metric) along $D_1,D_2,D_3$ as in Figure \ref{Figure1}. Since the submanifolds $P(2), \mathcal{N}_n(G),\mathcal{N}_n(G'),\mathcal{N}_n(G'_{i,j})$ of $P(3)$ intersect with $\partial  P(3)$ only in the interior of $D_1\cup D_2\cup D_3$, copies of these submanifolds are pasted together to closed submanifolds of $H_k$, and we denote them by $H_k(2), H_{k,n}(G), H_{k,n}(G'), $ $H_{k,n}(G'_{i,j})$, respectively. All of these closed submanifolds are contained in the interior of $H_k$, are homeomorphic to the genus-$k$ handlebody, and the complements of their interiors in $H_k$ are all homeomorphic to $\Sigma_k\times I$. Moreover, when $(i,j)$ runs over all pairs in $[-n,n]\cap \mathbb{Z}$, $H_{k,n}(G'_{i,j})$'s are disjoint from each other.

Since $H_k$ is obtained from $2k-2$ copies of $P(3)$ by pasting copies of $D_1,D_2,D_3$ via the identity, the piecewise smooth bi-Lipschitz homeomorphisms of $P(3)$ in Lemmas \ref{elementary}, \ref{translation}, and \ref{dependonn} induce piecewise smooth bi-Lipschitz self-homeomorphisms of $H_k$ that restrict to the identity on $\partial H_k$. These maps on $H_k$ further extend to piecewise smooth bi-Lipschitz self-homeomorphisms of $M_k$ by taking the identity map on $M_k\setminus \text{int}(H_k)$, and we get the following piecewise smooth bi-Lipschitz homeomorphisms of $M_k$.
\begin{enumerate}
\item By Lemma \ref{elementary}, we have a $C_1$-bi-Lipschitz homeomorphism $h_1:(M_k, H_k,H_{k,n}(G))\to (M_k, H_k,H_{k,n}(G'))$.
\item By Lemma \ref{translation}, for any $i,j\in [-n,n]\cap \mathbb{Z}$, we have a $C_2$-bi-Lipschitz homeomorphism $h_{2,i,j}:(M_k, H_k,H_{k,n}(G'))\to (M_k, H_k,H_{k,n}(G'_{i,j}))$.
\item By Lemma \ref{dependonn}, we have a $(C_3n)$-Lipschitz homeomorphism $h_3:(M_k, H_k,H_{k,n}(G))\to (M_k, H_k,H_k(2))$ such that $h_3^{-1}$ is $C_3$-Lipschitz.
\end{enumerate}

Now we are ready to prove Theorem \ref{technical}. 
\begin{proof}[Proof of Theorem \ref{technical}]
Roughly speaking, the map $F_n$ will restricts to the identity on the complement of $\cup_{i,j=-n}^nH_{k,n}(G'_{i,j})$, and it maps each $H_{k,n}(G'_{i,j})$ to its complement via a piecewise smooth orientation-reversing homeomorphism, with controlled Lipschitz constant. Since we have $(2n+1)^2$ terms in the union $\cup_{i,j=-n}^nH_{k,n}(G'_{i,j})$, we have $\text{deg}(F_n)=1-(2n+1)^2$.

Now we give details of the proof of Theorem \ref{technical}.
For $M_k=\#^kS^2\times S^1$, its standard genus-$k$ Heegaard decomposition $M_k=H\cup H'$ is obtained by taking two copies of the genus-$k$ handlebody and using the identity map to paste their boundaries. We take a Riemannian metric $g_0$ on $M$ such that its restriction to $H$ is isometric to $H_k$. By Lemma \ref{onemetric}, we only need to prove Theorem \ref{technical} for $g_0$.

Once we identify $H\subset M_k$ with the Riemannian genus-$k$ handlebody $H_k$, we get a submanifold $H_k(2)\subset H_k\subset M$. Since $H_k\setminus \text{int}(H_k(2))$ is homeomorphic to $\Sigma_k\times I$, $M_k=H_k(2)\cup (M_k\setminus \text{int}(H_k(2)))$ is also a standard genus-$k$ Heegaard decomposition of $M_k$. So there exists an orientation-reversing homeomorphism $h_4: M_k\to M_k$ that swaps the two components of $M_k\setminus \partial (H_k(2))$ and restricts to the identity on $\partial H_k(2)$. We can make $h_4$ to be a pieceswise smooth bi-Lipschitz homeomophism, with bi-Lipschitz constant $C_4>1$, which only depends on the Riemannian metric on $M_k$.

For any positive integer $n$, we define the desired map $F_n:M_k\to M_k$ by
\begin{equation*}
\begin{aligned}
F_n(x)=\begin{cases}
x &\text{if\ } x\in M_k\setminus \cup_{i,j=-n}^n\text{int}(H_{k,n}(G'_{i,j}))\\
h_{2,i,j}\circ h_1\circ h_3^{-1} \circ h_4 \circ h_3 \circ h_1^{-1}\circ h_{2,i,j}^{-1}(x) & \text{if\ } x\in H_{k,n}(G'_{i,j}) \text{\ for\ some\ }i,j\in [-n,n]\cap \mathbb{Z}.
\end{cases}
\end{aligned}
\end{equation*}
In the second case of the formula of $F_n(x)$, the composition works as the sequence below.
\begin{equation*}
\begin{aligned}
&H_{k,n}(G'_{i,j})\xrightarrow{h_{2,i,j}^{-1}} H_{k,n}(G')\xrightarrow{h_1^{-1}}H_{k,n}(G)\xrightarrow{h_3}H_k(2)\xrightarrow{h_4}M_k\setminus \text{int}(H_k(2))\xrightarrow{h_3^{-1}}\\
& M_k\setminus \text{int}(H_{k,n}(G))\xrightarrow{h_1} M_k\setminus \text{int}(H_{k,n}(G'))\xrightarrow{h_{2,i,j}}M_k\setminus \text{int}(H_{k,n}(G'_{i,j})).
\end{aligned}
\end{equation*}
This composition is an orientation-reversing homeomorphism $H_{k,n}(G'_{i,j})\to M_k\setminus \text{int}(H_{k,n}(G'_{i,j}))$ and we denote it by $F_{n,i,j}$.

Here we need the fact that $H_{k,n}(G'_{i,j})$'s are disjoint from each other to define $F_n$. Since $h_4$ restricts to the identity on $\partial H_k(2)$, the restriction of $F_{n,i,j}$ to $\partial H_{k,n}(G'_{i,j})$ is the identity. So $F_n$ is well-defined.

To compute $\text{deg}(F_n)$, we take any $x_0\in M_k \setminus \cup_{i,j=-n}^nH_{k,n}(G'_{i,j})$. Then $F_n^{-1}(x_0)=\{x_0\}\cup \{F_{n,i,j}^{-1}(x_0)\ |\ i,j\in [-n,n]\cap \mathbb{Z}\}$. For $F_n$, the local degree at $x_0$ is $1$, and the local degree at any $F_{n,i,j}^{-1}(x_0)$ is $-1$, since $F_{n,i,j}$ is an orientation reversing homeomorphism. Since $|[-n,n]\cap \mathbb{Z}|=2n+1$, we have $\text{deg}(F_n)=1-(2n+1)^2=-4n^2-4n$, thus item (1) of Theorem \ref{technical} holds.

By our estimates of bi-Lipschitz constants and the definition of $F_n$, $F_n$ is a Lipschitz map with Lipschitz constant $C_2\cdot C_1\cdot (C_3n)\cdot C_4\cdot C_3 \cdot C_1\cdot C_2=C_1^2C_2^2C_3^2C_4n$. We take $C=C_1^2C_2^2C_3^2C_4$, then item (2) of Theorem \ref{technical} holds.

Since $H'\subset M_k\setminus \cup_{i,j=-n}^n\text{int}(H_{k,n}(G'_{i,j}))$ and $F_n$ restricts to the identity on $M_k\setminus \cup_{i,j=-n}^n\text{int}(H_{k,n}(G'_{i,j}))$, we have $F_n|_{H'}=id_{H'}$. So item (3) of Theorem \ref{technical} holds. 

The proof of Theorem \ref{technical} is done.

\end{proof}

\bibliographystyle{amsalpha}

\begin{thebibliography}{GWWZ}

\bibitem[BGM]{BGM} A. Berdnikov, L. Guth, F. Manin, 
{\it Degrees of maps and multiscale geometry.}
Forum Math. Pi 12 (2024), Paper No. e2, 48 pp.

\bibitem[Be]{Be} R. Bethuel {\it The approximation problem for Sobolev maps between two manifolds},
Acta Math. 167 (1991), no. 3-4, 153--206.

\bibitem[DLWWW]{DLWWW} J. R. Duan, J. F. Lin,  S.C. Wang Z. Z. Wang,  D. Y. Wei,  {\it Flexible exponent of geometric 3-manifolds,  Legendrian maps of Seifert spaces.} Preprint 2026

\bibitem[Gr]{Gr} M. Gromov, {\it Volume and bounded cohomology},  Inst. Hautes \'{E}tudes Sci. Publ. Math. No. 56, (1982), 5--99

\bibitem[Ha]{Ha} P. Hajlasz, {\it Approximation of Sobolev mappings}, 
Nonlinear Anal. 22 (1994), no. 12, 1579--1591.




\bibitem[KSSV]{KSSV} M. Kunzinger, H. Schichl, R. Steinbauer, J. Vickers, {\it Global Gronwall estimates for integral curves on Riemannian manifolds}, Rev. Mat. Complut. 19 (2006), no. 1, 133--137.









\bibitem[ScW]{ScW} P. Scott, C. T. C. Wall: {\it Topological methods in group theory.
Homological group theory}, (Proc. Sympos., Durham (1979), 137--203),
London Math. Soc. Lecture Note Ser. 36, Cambridge Univ. Press,
Cambridge-New York, 1979.


\bibitem[Thom]{Thom}
R. Thom, {\em Quelques propri\'et\'es globales des vari\'et\'es differentiable}, Comm. Math. Helv. {\bf 28}, (1954), 17--86.


\bibitem[Th]{Th} {\sc W.P. Thurston},
{\it The geometry and topology of $3$-manifolds}, Lecture Notes, Princeton 1977.

\bibitem[Wang]{Wang} S. C. Wang, {\it The  $\pi_1$-injectivity of self-maps of nonzero degree on  3-manifolds.}
Math. Ann. 297 (1993), no. 1, 171--189.










\end{thebibliography}

\end{document}